\documentclass[a4paper,12pt,reqno]{amsart}
\usepackage{tikz-cd}
\usepackage[utf8]{inputenc}
\usepackage{amsmath,amssymb}%
\usepackage{exscale}%
\usepackage{amsfonts}%
\usepackage{palatino}
\usepackage{amsthm}
\usepackage{epsfig}
\usepackage{xcolor}
\usepackage{hyperref}
\usepackage[english]{babel}
\usepackage{mathtools}
\usepackage{cite}
\usepackage[ left=2.8cm,
        right=2.8cm,
        top=3.5cm,
        bottom=3.5cm]{geometry}

\usepackage{hyperref}
\usepackage{cleveref}
\usepackage[footnote,draft,silent,nomargin]{fixme}

\hypersetup{%
	plainpages=false,%
	pdfauthor={Author(s)},%
	pdftitle={Title},%
	pdfsubject={Subject},%
	bookmarksnumbered=true,%
	colorlinks,%
	citecolor=black,%
	filecolor=black,%
	linkcolor=black,
	urlcolor=black,%
	pdfstartview=FitH%
}

\let\epsilon\varepsilon

\let\phi\varphi

\let\theta\vartheta

\newtheorem{mytheorem}{Theorem}[section]
\newtheorem{myprop}[mytheorem]{Proposition}
\newtheorem{mycor}[mytheorem]{Corollary} 
\theoremstyle{definition}
\newtheorem{mydef}[mytheorem]{Definition}

\newtheorem{myre}[mytheorem]{Remark}
\newtheorem{mylemma}[mytheorem]{Lemma}

\newcommand{\R}{\mathbb{R}}
\newcommand{\Z}{\mathbb{Z}}

\newcommand{\ip}[1]{{\left\langle #1\right\rangle}}

\newcommand{\norm}[1]{\left\lVert #1\right\rVert}

\renewcommand{\vec}[1]{\mathbf{#1}}

\DeclareMathOperator{\e}{e}
\renewcommand{\d}{\mbox{d}}
\renewcommand{\epsilon}{\varepsilon}

\begin{document}
\title[On homogeneous decomposition spaces]{On homogeneous decomposition spaces and associated decompositions of distribution spaces} \author{Zeineb Al-Jawahri and Morten Nielsen }   \date{\today}
\begin{abstract}
A new construction of decomposition smoothness spaces of homogeneous type is considered. The smoothness spaces  are based on structured and flexible decompositions of the frequency space $\mathbb{R}^d\backslash\{0\}$. We construct simple adapted tight frames for $L_2(\mathbb{R}^d)$ that can be used to fully characterise the smoothness norm in terms of a sparseness condition imposed on the frame coefficients. 
Moreover, it is proved that the frames  provide a universal decomposition of tempered distributions with convergence in the tempered distributions modulo polynomials. As an application of the general theory, the notion of homogeneous $\alpha$-modulation spaces is introduced.
\end{abstract}
\maketitle

\section{Introduction}
A major benefit of using smooth wavelet bases is the fact that they are universally applicable in the sense that any tempered distribution has an expansion in the basis. That is, any suitable function or tempered distribution can be decomposed in a corresponding wavelet series with convergence at least in the sense of tempered distributions (possibly modulo polynomials), and the coefficients of the wavelet series precisely capture the smoothness properties of the function or distribution, see \cite{MR1070037}. For example, it is known, see \cite{Meyer,MR1070037}, that suitable sparseness of a wavelet expansion is equivalent to smoothness measured in a Besov space. The fact that smoothness leads to sparse expansions has many important practical implications, e.g. for signal compression, where it is possible to use a sparse representation of a function to compress that function simply by thresholding the expansion coefficients. Several now classical applications have shown that wavelets are well suited to compress images with smoothness measured in a suitable Besov space, see e.g.\ \cite{MR1162221,MR1175690}. 

While wavelet systems and related techniques such as the $\varphi$-transform, see \cite{frazier85,MR1070037}, are based on dyadic decompositions of the frequency domain $\mathbb{R}^d$, the point we would like to make in the present paper is that the dyadic decomposition itself does not play {\em the} decisive role for obtaining nice universal decompositions of tempered distributions capturing smoothness in the associated expansion coefficients.
One of the main contributions of this article shows that fairly general but structured decompositions of the frequency domain allow for any function or tempered distribution to be expanded in an associated compatible system. The systems constructed have atoms that form frames and combines features of both Gabor and wavelet systems reflecting the chosen decomposition of the frequency domain. With this more general frame system we are able to analyse any function or tempered distribution by examining the corresponding frame coefficients. 
 
There has been considerable recent interest in analysing general representation systems and associated general notions of smoothness, see e.g.\  \cite{MR3016517,MR3048591,MR3345605,MR2543193} and references therein. This is in part motivated by applications to signal and image processing of generalized wavelet systems such as $\alpha$-modulation frames, see \cite{MR2396847,MR2737936,fornasier07,MR3520717}, and Shearlet type systems, see \cite{MR3016517,MR2861594,MR3023401,MR3048591}.
The second generation systems are generally based on modified decompositions of the frequency domain as compared to the classical dyadic decompositions. One important implication is, however, that one cannot directly capture coefficient sparseness in terms of smoothness measured on the classical Besov scale associated since it is linked directly to a dyadic frequency decomposition. This problem can be overcome by replacing the Besov scale by decomposition smoothness spaces, see e.g.\ \cite{grobner,MR3048591,BN08,Nielsen}.

The general  theory of decomposition spaces was introduced by Feichtinger and Gr\"obner, see \cite{MR809337,MR910054}, where spaces based on decomposition of both the time and frequency domain were considered.  Triebel \cite{MR725159}  showed that modulation spaces are compatible with the decomposition approach. This later inspired a more general treatment of decomposition smoothness spaces  \cite{BN08,Nielsen}. In the same spirit, very general homogeneous (anisotropic) Besov and Triebel-Lizorkin spaces were considered by Bownik \cite{MR2179611} and by Bownik and Ho  \cite{BH05}.  In a similar dyadic setup, a general approach to homogeneous spaces has  been studied in detail recently by Triebel \cite{MR3409094,MR3682619}.

A benefit of the decomposition space approach is also that anisotropic spaces can be considered without any additional technical complications, but 
one specific shortcoming of the decomposition space theory developed in \cite{BN08,Nielsen} is that it does not allow one to decompose general tempered distributions. Moreover, it is not immediately clear how to treat the case of  homogeneous smoothness spaces with frequency domain $\mathbb{R}^d\backslash\{0\}$. These issues will be addressed in the present paper, where we will consider a general construction of homogeneous smoothness spaces defined on $\mathbb{R}^d\backslash\{0\}$. The family of spaces which we consider are based on structured decomposition of the frequency space $\mathbb{R}^d\backslash\{0\}$. This way of defining smoothness spaces makes it possible to find adapted tight frames for $L_2(\mathbb{R}^d)$. Such frames turn out to provide universal decompositions of tempered distributions with convergence in the tempered distributions modulo polynomials. Moreover, atomic decompositions of the corresponding homogeneous smoothness spaces are obtained, and the smoothness spaces can be completely characterized by a sparseness condition on the frame coefficients. This makes it possible to compress the elements of such homogeneous smoothness spaces using the frame. 

Another approach to homogeneous decomposition type smoothness spaces is to use the framework of Coorbit-theory, see \cite{MR1021139}.  In the Coorbit-setting,  homogeneous spaces and associated stable expansions have been studied by Voigtlaender \cite{2016arXiv160102201V,Voigtlaender:564979} and by F\"uhr and Voigtlaender  \cite{MR3345605}. One marked difference between the Coorbit-approach  and the approach presented in this paper is that our coverings are not constrained by any group theoretic structure.  

The outline of this paper is as follows. 
In Section \ref{sec:background} we introduce the general setting of the paper. We begin by recalling basic definitions and properties of anisotropic spaces and structured admissible coverings, and we close the section by introducing a so-called hybrid regulation function.  
In Section \ref{sec:construction} we construct structured admissible coverings of $\mathbb{R}^d\backslash\{0\}$. The construction is generated using open (anisotropic) balls in $\mathbb{R}^d\backslash\{0\}$. 
In Section \ref{sec:hmspaces} we provide the definition of homogeneous decomposition spaces and state some fundamental properties of these spaces.
The main contribution of the paper is found in Section \ref{sec:conv}, where we show that any tempered distribution can be decomposed into a sum of compactly supported functions satisfying the conditions stated in Proposition \ref{prop1}. 
In Section \ref{section:frames} we first construct a tight frame for $L_2(\mathbb{R}^d)$ compatible with the structured admissible coverings. We then provide the fundamental reproducing identity for tempered distributions by means of the frame system. 
In Section \ref{sec:char} we show that the frame system gives an atomic decomposition of the homogeneous decomposition spaces. 
In Section \ref{sec:examples} we provide an example on constructing a hybrid regulation function and use the methods presented in Section \ref{sec:construction} to obtain structured admissible coverings of $\mathbb{R}^d\backslash\{0\}$. We show that the homogeneous Besov space corresponds to a special case of the construction and we define a new class of homogeneous anisotropic $\alpha$-modulation spaces yielding a new homogeneous version of the $\alpha$-modulation spaces introduced by P.\ Gr\"obner \cite{grobner}.
Finally, in Appendix \ref{sec:appA} we prove some technical lemmas used throughout the paper and Appendix \ref{sec:appB} contains the proof of the fundamental properties of homogeneous decomposition spaces. 

Let us fix some of the notation used in this paper. We let $\mathcal{F}(f)(\xi) := \hat{f}(\xi) := (2\pi)^{-d/2}\int_{\mathbb{R}^d} f(x) \e^{-ix\cdot \xi} \d x, f\in L_1(\mathbb{R}^d)$ denote the Fourier transform with the usual extension to $L_2(\mathbb{R}^d)$. By $F \asymp G$ we mean that there exists two constants $0<c_1\leq c_2<\infty$ such that $c_1F\leq G \leq c_2F$. For two normed vector spaces $X,Y$ we mean by $X\hookrightarrow Y$ that $X\subset Y$ and $\norm{f}_Y \leq C \norm{f}_X$ for some $C>0$ and for all $f\in X$.

\section{The General Setting}\label{sec:background}
In this section we introduce the notation needed to define homogeneous decomposition spaces. The terminology is to a large degree inherited from Feichtinger and Gr\"obner, see \cite{MR809337,MR910054}, and from \cite{Nielsen}. However, certain modifications have been made to adapt to the homogeneous setup. In particular, in Section \ref{sec:hybrid} we introduce a new notion of a so-called hybrid regulation function needed to generate suitable coverings of the frequency space $\mathbb{R}^d\backslash\{0\}$.
\subsection{Anisotropic Norm}
Since we will be working with anisotropic spaces, we first provide a definition of the anisotropic norm $|\cdot|_{\vec{a}}$. 
Let $|\cdot|$ denote the Euclidean norm on $\mathbb{R}^d$ induced by the inner product $\ip{\cdot,\cdot}$ and let $\vec{a} = (a_1,\ldots,a_d)$ be an anisotropy on $\mathbb{R}^d$ satisfying $a_i > 0$ and $\sum_{i=1}^d a_i = d$. 
For $t>0$, the anisotropic dilation matrix $D_{\vec{a}}(t)$ is given by $D_{\vec{a}}(t) := \text{diag}(t^{a_1},\ldots,t^{a_d})$. Based on \cite[Proposition 1.7]{Stein} we may state the following definition.

\begin{mydef}\label{def:aninorm} 
We define the function $|\cdot|_{\vec{a}} : \R^d\to \R_+$ by setting $|0|_{\vec{a}} := 0$ and for $\xi\in\R^d\backslash\{0\}$ we set $|\xi|_{\vec{a}} = t$, where $t$ is the unique solution to the equation $\left|D_{\vec{a}}(1/t)\xi\right| = 1$.
\end{mydef}

According to \cite{Stein} we have the following standard properties of $|\cdot|_{\vec{a}}$:

\begin{itemize}
	\item[(1)] $|\cdot|_{\vec{a}} \in C^\infty(\mathbb{R}^d\backslash\{0\})$.
	\item[(2)] There exists a constant $K\geq 1$ such that
	\begin{equation*}
|\xi + \zeta|_{\vec{a}} \leq K(|\xi|_{\vec{a}}+|\zeta|_{\vec{a}}), \quad \forall \; \xi,\zeta\in\mathbb{R}^d\backslash\{0\}.
	\end{equation*}
\item[(3)] For $t>0$
\begin{equation}\label{eq:conditionaninorm}
|D_{\vec{a}}(t)\xi|_{\vec{a}} = t|\xi|_{\vec{a}}
\end{equation}
\item[(4)] For $\xi \in \mathbb{R}^d\backslash\{0\}$
\begin{align}\label{eq:anormJ2}
c_1 |\xi|^{\alpha_1} &\leq |\xi|_{\vec{a}} \leq c_2 |\xi|^{\alpha_2}, \quad \text{if} \; |\xi|_{\vec{a}} \geq 1, \; \text{and} \\
\label{eq:anormJ1} c_3 |\xi|^{\alpha_2} &\leq |\xi|_{\vec{a}} \leq c_4 |\xi|^{\alpha_1}, \quad \text{if} \; |\xi|_{\vec{a}} < 1,
\end{align}
where $\alpha_1$ denotes the smallest, and $\alpha_2$ the largest entry in $\vec{a} = (a_1,\ldots,a_d)$.
\end{itemize}

The anisotropic norm $|\cdot|_{\vec{a}}$ from Definition \ref{def:aninorm} induces a quasi-distance $\d: \mathbb{R}^d \times \mathbb{R}^d \to [0,\infty)$ given by $\d(\xi,\zeta) := |\xi-\zeta|_{\vec{a}}$. Notice that $(\mathbb{R}^d\backslash\{0\},\d,\d \xi)$ is a space of homogeneous type.

\subsection{Structured Admissible Coverings}\label{sec:homogenems}
In this subsection we recall the notion of structured admissible coverings, see \cite{MR1021139,MR910054}. We say that a collection $\mathcal{Q} := \{Q_j\}_{j\in J}$ of measurable subsets in $\mathbb{R}^d\backslash\{0\}$ is an admissible covering of $\mathbb{R}^d\backslash\{0\}$ if $\mathbb{R}^d\backslash\{0\} = \cup_{j\in J} Q_j$ and $\mathcal{Q}$ satisfy the finite overlap property. That is, if we define
\begin{equation*}
\tilde{j} := \{i\in J: Q_i\cap Q_j \ne \emptyset\}
\end{equation*}
to be the set of neighbours of $Q_j$, then the number of neighbours of each set is uniformly bounded. The structured coverings are a special class of admissible coverings. The idea is, essentially, that each set $Q_j$ is an affine image $Q_j := T_j(Q) := A_jQ + b_j$ of a fixed set $Q \subset \mathbb{R}^d\backslash\{0\}$, where $A \in GL(d,\mathbb{R})$ is an invertible matrix and $b\in \mathbb{R}^d$. We note that the coverings we consider are of the frequency domain $\mathbb{R}^d\backslash\{0\}$. 

\begin{mydef}\label{def:sac}
	Let $J \ne \emptyset$ be a countable index set. Given a family $\mathcal{T} := \{T_j\}_{j\in J} := \{A_j \cdot + b_j\}_{j\in J}$ of invertible affine transformations on $\mathbb{R}^d$. Suppose there exists two bounded open sets $P\subset Q \subset \mathbb{R}^d\backslash\{0\}$, with $P$ compactly contained in $Q$, such that
	\begin{equation*}
		\{P_{j}\}_{j\in J} \quad \text{and} \quad \{Q_{j}\}_{j\in J} \qquad \text{are admissible coverings of $\mathbb{R}^d\backslash\{0\}$.}
	\end{equation*}
	Also assume that there exists a constant $K$ such that
	\begin{equation}\label{prop:eq:bounded}
		(A_jQ + b_j) \cap (A_{k}Q+b_{k}) \ne \emptyset \Rightarrow \norm{A_{k}^{-1}A_j}_{\ell_\infty(\mathbb{R}^d\times \mathbb{R}^d)} \leq K.
	\end{equation}
	Then we call $\mathcal{Q} := \{Q_j\}_{j\in J}$ a structured admissible covering and $\{T_j\}_{j\in J}$ a structured family of affine transformations.
\end{mydef} 

For later use it will be convenient to define $|T| := |A| := |\det A|$. 
We also need partitions of unity compatible with the coverings from Definition \ref{def:sac}.

\begin{mydef}\label{def:bapu}
	Let $\mathcal{Q} := \{Q_j\}_{j\in J}$ be a structured admissible covering of $\mathbb{R}^d\backslash\{0\}$. A family $\Psi = \{\psi_j\}_{j\in J}$ of functions on $\mathbb{R}^d\backslash\{0\}$ is called a bounded admissible partition of unity (BAPU) for $\mathcal{Q}$ if
	\begin{itemize}
		\item[(1)] $\text{supp}(\psi_j) \subset Q_j$ for all $j \in J$,
		\item[(2)] $\sum_{j\in J}\psi_j(\xi) = 1$ for all $\xi \in \mathbb{R}^d\backslash\{0\}$,
		\item[(3)] the constant $C_p := \sup_{j\in J} |Q_j|^{1/p-1} \norm{\mathcal{F}^{-1}\psi_j}_{L_p}$ is finite for all $p \in (0,1]$.
	\end{itemize}
\end{mydef}

\begin{myre}
The constant $C_p$ appearing in the third condition of Definition \ref{def:bapu} is required to be finite for all $0<p\leq 1$ since we  define homogeneous decomposition spaces below  based on the full range of $L_p(\R^d)$-spaces, $0<p\leq \infty$.  However, if one restricts attention to $L_p$-based spaces in the Banach regime, when $1\leq p \leq \infty$, then only the ``endpoint'' condition, $C_1<\infty$, is needed.
\end{myre}

Given $\psi_j \in \Psi$, we define the Fourier multiplier $\psi_j(D)f := \mathcal{F}^{-1}(\psi_j\mathcal{F}f)$ for $f\in L_2(\mathbb{R}^d)$. By \cite[Proposition 1.5.1]{HansTriebel} the conditions in Definition \ref{def:bapu} ensure that all band-limited functions $f\in L_p(\mathbb{R}^d), 0<p\leq \infty$ satisfy
\begin{equation}\label{eq:bandlimited}
	\norm{\psi_j(D)f}_{L_p} \leq C \norm{f}_{L_p}, \quad \forall \; j\in J.
\end{equation}
For a BAPU $\{\psi_j\}_{j\in J}$ associated with a structured admissible covering $\{Q_j\}_{j\in J}$ we define	$\widetilde{\psi}_j := \sum_{k \in \tilde{j}} \psi_k.$ For later use, we observe that $\psi_j\widetilde{\psi}_j = \psi_j$. \\

An important consequence of Definition \ref{def:sac} is that any structured admissible covering admits a bounded admissible partition of unity.

\begin{myprop}\label{prop1}
Let $\mathcal{Q} = \{Q_j\}_{j\in J} := \{T_jQ\}_{j\in J}$ be a structured admissible covering of $\mathbb{R}^d\backslash\{0\}$ and $\{T_j\}_{j\in J}$ a structured family of affine transformations. Then there exist
	\begin{itemize}
		\item[(i)] A BAPU $\{\psi_j\}_{j\in J} \subset \mathcal{S}(\mathbb{R}^d)$ corresponding to $\{Q_j\}_{j\in J}$.
		\item[(ii)] A system $\{\phi_j\}_{j\in J} \subset \mathcal{S}(\mathbb{R}^d)$ satisfying
				\end{itemize}
\begin{align}
\text{\textbullet} \; \label{eq:phisatisfy1} &\text{supp}(\phi_j) \subset Q_j, \quad \forall j \in J, \\
\text{\textbullet} \; \label{eq:phisatisfy2} &\sum_{j\in J} \phi_j^2(\xi) = 1, \quad \forall \xi \in \mathbb{R}^d\backslash\{0\}, \\
\text{\textbullet} \; &C_p := \sup_{j\in J} |Q_j|^{1/p-1} \norm{\mathcal{F}^{-1}\phi_j}_{L_p} < \infty, \quad \forall p \in (0,1] \nonumber.
\end{align}

\end{myprop}

\begin{proof}
The proof of Proposition \ref{prop1} follows similar techniques as in \cite[Proposition 1]{Nielsen}. Notice that we have the equivalence $|P_j| \asymp |T_j| \asymp |Q_j|$ uniformly in $j \in J$. We begin by proving (i). To construct a BAPU corresponding to the admissible covering $\{Q_j\}_{j \in J}$ we proceed as follows: Choose a non-negative function $\Phi \in C^\infty(\mathbb{R}^d)$ with $\Phi(\xi) = 1$ for all $\xi \in P$ and $\text{supp}(\Phi) \subset Q$. For  $j\in J$, let $g_j(\xi) := \Phi(T_j^{-1}\xi)$. Then $g_j \in C^\infty(\mathbb{R}^d)\backslash\{0\}$ with $P_j \subset \text{supp}(g_j) \subset Q_j$. 
Now, consider the (locally finite) sum $g(\xi) := \sum_{j\in J} g_j(\xi)$. Since $\{Q_j\}_{j \in J}$ has finite height and $\{P_j\}_{j \in J}$ covers $\mathbb{R}^d\backslash\{0\}$ we have $1 \leq g(\xi) \leq N<\infty$ for all $\xi \in \mathbb{R}^d\backslash\{0\}$. For each $j \in J$ we can therefore define 
	\begin{equation} \label{eq:bapu}
	\psi_j(\xi) :=
	\frac{g_j(\xi)}{g(\xi)} =
	\frac{\Phi(T_j^{-1}\xi)}{\sum_{j\in J} \Phi(T_j^{-1}\xi)}, \quad \xi \in \mathbb{R}^d\backslash\{0\}.
	\end{equation}
	It follows that the function in \eqref{eq:bapu} satisfies the first two properties in Definition \ref{def:bapu}. 
	Thus, in order to conclude we need to verify that $\sup_j |T_j|^{1/p-1} \norm{\mathcal{F}^{-1}\psi_j}_{L_p} < \infty$, for all $p \in (0,1]$. 
	For $j \in J$ we define	\begin{equation}\label{eq:hfunction}
	h_j(\xi) := \psi_j(T_j\xi) = \frac{\Phi(\xi)}{g(T_j\xi)}.
	\end{equation}
	Setting $f= \mathcal{F}^{-1}h_j$ in Lemma \ref{lemma12} we have $\hat{f}_j(\xi) := h_j(T_j^{-1}\xi)$ and $f_j(\xi) = \mathcal{F}^{-1}\psi_j(\xi)$. The lemma now implies $\norm{\mathcal{F}^{-1}\psi_j}_{L_p} = |T_j|^{1-1/p} \norm{\mathcal{F}^{-1}h_j}_{L_p}$ for $0<p\leq \infty$. 
Let $T_j = A_j\cdot + b_j$. Expanding the expression in \eqref{eq:hfunction} yields
	\begingroup
	\addtolength{\jot}{3mm}
	\begin{align*}
		h_j(\xi) &= \frac{\Phi(\xi)}{\sum_{L\in\mathcal{T}} \Phi(L^{-1}T_j\xi)} =
		\frac{\Phi(\xi)}{\sum_{L\in\mathcal{T}} \Phi(L^{-1}(A_j\xi + b_j))} &\\
		&= \frac{\Phi(\xi)}{\sum_{k\in J} \Phi(A_{k}^{-1}A_j\xi + A_{k}^{-1}b_j- A_{k}^{-1}b_{k})}. 
	\end{align*}
	\endgroup
	We consider $\partial^\beta h_j$. Let us first recall the formula 
	\begin{equation}\label{eq:Leib}
	\left| \partial^\beta \frac{u(\xi)}{v(\xi)} \right| = \sum_{\substack{|\lambda|\leq |\beta| \\ k=1,\ldots,2|\beta|}} C_k^{\lambda,\beta} \left| \frac{ \partial^\lambda u \partial^{\beta-\lambda} v}{v^{k}} \right|,
	\end{equation}
	for certain constants $C_k^{\lambda,\beta}$. Let $u(\xi) := \Phi(\xi)$ and $v(\xi) := \Phi(A_{k}^{-1}A_j\xi - A_{k}^{-1}b_{k} + A_{k}^{-1}b_j)$. Since \eqref{eq:Leib} represents a finite sum, and the denominator $v(\xi)$ is a smooth function which satisfies $1 \leq v(\xi) \leq \tilde{N}$ for $\tilde{N} < \infty$ it suffices to consider the numerator on the right hand side of \eqref{eq:Leib}. We consider $\partial^\eta v$, where $\eta := \beta-\lambda$. An application of the chain rule shows that $\partial^\eta v = \sum_{\beta: |\beta| = |\eta|} p_\beta \partial^\beta \Phi$, where $p_\beta$ are monomials of degree $|\eta|$ in the entries of $A_{k}^{-1}A_j$. Combining this with the estimate in \eqref{prop:eq:bounded} and the fact that $\text{supp}(\Phi) \subset Q$ we have
	\begin{equation*}
|\partial^\beta h_j(\xi)| \leq C_{\beta} K^{|\beta|} \chi_{Q}(\xi), \quad \beta \in \mathbb{N}_0^d,
	\end{equation*}
	where $C_\beta$ is independent of $j \in J$. 
Since $h_j$ has compact support in $Q_j$ the integration by parts formula yields
\begin{align*}
	\norm{\mathcal{F}^{-1}h_j}^p_{L_p} 
	&= \int_{\mathbb{R}^d} \left| (2\pi)^{-d/2} \int_{\mathbb{R}^d} h_j(\xi) \e^{ix\cdot \xi} \d \xi \right|^p \d x \\
	&\leq C_d\left( \sum_{|\beta| \leq \lceil(d+1)/p\rceil} \norm{\partial^\beta h_j}_{L_1} \right)^p \int_{\mathbb{R}^d} (1+|x|)^{-d-1} \d x < \infty.
\end{align*}
This shows that the last property in Definition \ref{def:bapu} is satisfied, hence proving that $\{\psi_j\}_{j\in J}$ is a BAPU corresponding to the admissible covering $\{Q_j\}_{j\in J}$. By defining 
\begin{equation}\label{eq:squarerootbapu}
\phi_j(\xi) := \frac{g_j(\xi)}{\sqrt{\sum_{k\in J} g_{k}^2(\xi)}}, \quad j\in J,
\end{equation}
we may use the same arguments as above to conclude that $\{\phi_j\}_{j\in J}$ satisfies the properties stated in Proposition \ref{prop1}, hence proving (ii). \qedhere
\end{proof}

The system $\{\phi_j\}_{j\in J}$ in Proposition \ref{prop1} (ii), which in a sense defines the ''square root'' of a BAPU, will be used in the definition of homogeneous decomposition spaces. 
However, before we can provide this definition we need to introduce a so-called hybrid regulation function $\tilde{h}$ satisfying some growth conditions, needed to ensure completeness of the homogeneous decomposition spaces. 

\subsection{Hybrid Regulation Functions}\label{sec:hybrid}
One way to construct  BAPUs suitable for the construction of decomposition spaces is to define BAPUs  with support contained in anisotropic balls of the type $\{B(\xi_j,h(\xi_j))\}$, where $h$ is a so-called regulation function or weight. This approach was first considered by Feichtinger  \cite{MR910054} and later used to construct new inhomogeneous smoothness spaces, see  \cite{Nielsen,grobner}.
General properties of weights and regulation functions 
were studied by Feichtinger \cite{MR599884} and Gr\"ochenig  \cite{MR2385335}.

Typical examples of regulation functions are $h_\alpha(\xi):=|\xi|_{\mathbf{a}}^\alpha$ for $0\leq \alpha\leq 1$. However, this quite natural approach needs to be modified in a homogeneous setup due to the fact that $B(\xi_j,h(\xi_j))$, in general, naturally ``covers'' the excluded frequency $0$ when $|\xi|_{\mathbf{a}}\approx 0$, since the prototype regulation functions generally have ``too large'' values for small frequencies. To fix this problem, we introduce the notion of a hybrid regulation function.
We begin by recalling  the notion of $\d$-moderateness. 

\begin{mydef}
Let $(\mathbb{R}^d\backslash\{0\},\d,\d \xi)$ be a space of homogeneous type. A function $h:\mathbb{R}^d\backslash\{0\} \to (0,\infty)$ is called $\d$-moderate if there exist constants $R,\delta_0 > 0$ such that $\d(\xi,\zeta) \leq \delta_0 h(\xi)$ implies $R^{-1} \leq h(\xi)/h(\zeta) \leq R$ for all $\xi,\zeta\in\mathbb{R}^d\backslash\{0\}$. 
\end{mydef}
We now provide the definition of a hybrid regulation function.

\begin{mydef} \label{def:hfunction}
	Take a non-negative ramp function $\rho \in \mathbb{C}^s$ for some $s\geq 1$ satisfying
	\begin{equation}\label{eq:rho}	
	\rho(\xi) = \begin{cases}
	1& \text{for $0<|\xi|_{\vec{a}} \leq \frac{2}{3}$} \\ 
	\\
	0& \text{for $|\xi|_{\vec{a}} \geq \frac{4}{3}$}
	\end{cases}
	\end{equation}
	and define $\tilde{h}: \mathbb{R}^d\backslash\{0\} \to (0,\infty)$ as
	\begin{equation*}
	\tilde{h}(\xi) = \rho h_1(\xi) + (1-\rho)h_2(\xi),
	\end{equation*}
	where $h_1(\xi)$ and $h_2(\xi)$ are both $\d$-moderate functions satisfying
	\begin{subequations}
		\begin{align}\label{eq:h1}
			c_0|\xi|_{\vec{a}}^{r} \leq h_1(\xi) &\leq c_1|\xi|_{\vec{a}}, &\text{for
				some $c_0, c_1>0$ and $r\geq 1$,} \\ \intertext{and}
			\label{eq:h2}c_2 \leq h_2(\xi) &\leq c_3|\xi|_{\vec{a}}, &\text{for some $c_2, c_3 > 0$.}
		\end{align}
	\end{subequations}
	We call $\tilde{h}: \mathbb{R}^d\backslash\{0\} \to (0,\infty)$ a \textit{hybrid regulation function}.
\end{mydef}
We will provide a specific example on constructing a hybrid regulation function in Section \ref{sec:examples}. The following lemma shows that a hybrid regulation function satisfies the condition of $\d$-moderateness.

\begin{mylemma}\label{lemma1}
	Consider $(\mathbb{R}^d\backslash\{0\},\d,\d \xi)$ and let $\tilde{h}: \mathbb{R}^d\backslash\{0\} \to (0,\infty)$ be a hybrid regulation function in the sense of Definition \ref{def:hfunction}. 
	Then $\tilde{h}$ is d-moderate. 
\end{mylemma}

\begin{proof}
	Since $\tilde{h}$ is a hybrid regulation function, the functions $h_1, h_2$ are assumed to be $\d$-moderate. Thus, let $\delta_0^1, R_1$ and $\delta_0^2, R_2$ be the constants associated to the $\d$-moderation of $h_1, h_2$, respectively. We will consider three different cases for $\xi,\zeta \in \mathbb{R}^d\backslash\{0\}$.  
	
	\begin{itemize}
		\item[Case 1:] 
		Let $0 < |\xi|_{\vec{a}} < C < \tfrac{2}{3}$ for a suitable $C$ to be specified later. Choose $\delta_0 > 0$ such that $\delta_0 \leq \min\{\delta_0^1,\delta_0^2\}$ and $\delta_0 c_1 < 1$. Let
		\begin{equation*}
		|\xi-\zeta|_{\vec{a}} \leq \delta_0 h_1(\xi) \leq c_1\delta_0|\xi|_{\vec{a}} \leq c_1\delta_0 C.
		\end{equation*}
		Let $K \geq 1$ be the constant from the quasi triangle inequality for $|\cdot|_{\vec{a}}$. Then 
		\begin{equation*}
		|\zeta|_{\vec{a}} = |\xi+(\zeta-\xi)|_{\vec{a}} \leq K(|\xi|_{\vec{a}} + |\zeta - \xi|_{\vec{a}}) \leq K(C+c_1\delta_0 C).
		\end{equation*}
		Set $C < \frac{2/3}{K(1+c_1\delta_0)}$, then $|\zeta|_{\vec{a}} \leq \frac{2}{3}$. Thus, we have
		\begin{equation*}
		R_1^{-1} \leq \frac{\tilde{h}(\xi)}{\tilde{h}(\zeta)} = \frac{h_1(\xi)}{h_1(\zeta)} \leq R_1.
		\end{equation*}
		
		\item[Case 2:] Let $C \leq |\xi|_{\vec{a}} \leq C_2$ for $C_2 \geq \tfrac{8}{3}K$. This particular choice of $C_2$ will be justified in Case 3 below. 
		Since the functions $h_1(\xi)$ and $h_2(\xi)$ are continuous on $\mathbb{R}^d\backslash\{0\}$ there exists constants $M_1,M_2$ such that $M_1^{-1} \leq h_1, h_2 \leq M_2$ on $C\leq |\xi|_{\vec{a}} \leq C_2$. 
		Choose $\delta_0 > 0$ such that $\delta_0 \leq \frac{C}{\max\{c_1,c_3\}C_22K}$ and let
		\begin{equation*}
		|\xi-\zeta|_{\vec{a}} \leq \delta_0 \tilde{h}(\xi) \leq \delta_0 \max\{c_1,c_3\}|\xi|_{\vec{a}} \leq \delta_0 \max\{c_1,c_3\} C_2.
		\end{equation*}
		Then
		\begin{align*}
		|\zeta|_{\vec{a}} &\leq K(|\xi|_{\vec{a}} + |\zeta-\xi|_{\vec{a}}) \leq K(C_2 + \delta_0 \max\{c_1,c_3\} C_2) 
		\leq K\left(C_2 + \frac{C}{2K}\right),
		\end{align*}
and
		\begin{align*}
		|\xi|_{\vec{a}} &\leq K(|\xi-\zeta|_{\vec{a}} + |\zeta|_{\vec{a}}) \quad \Leftrightarrow \\ |\zeta|_{\vec{a}} &\geq \frac{|\xi|_{\vec{a}}}{K} - |\xi-\zeta|_{\vec{a}} 
		\geq \frac{C}{K} - \delta_0 \max\{c_1,c_3\}C_2 
		\geq \frac{C}{2K}. 
		\end{align*}
		Thus, for $\widetilde{C} \leq \frac{C}{2K}$ and $\widetilde{C}_2 \geq  K\left(C_2 + \frac{C}{2K}\right)$ we have $\widetilde{C} \leq |\zeta|_{\vec{a}} \leq \widetilde{C}_2$. 
		By the same arguments as above there exists constants $\widetilde{M}_1, \widetilde{M}_2$ such that $\widetilde{M}_1^{-1} \leq h_1, h_2 \leq \widetilde{M}_2$ on $\widetilde{C} \leq |\zeta|_{\vec{a}} \leq \widetilde{C}_2$.
		Now choose $M=\max\{M_1,M_2,\widetilde{M}_1,\widetilde{M}_2\}$. Then
		\begin{equation*}
		\frac{\tilde{h}(\xi)}{\tilde{h}(\zeta)} = \frac{ \rho h_1(\xi) + (1-\rho)h_2(\xi)}{\rho h_1(\zeta) + (1-\rho)h_2(\zeta)} \leq \frac{\rho M + (1-\rho)M}{\rho M^{-1} + (1-\rho)M^{-1}} = M^2,
		\end{equation*}
		and similarly
		\begin{equation*}
		\frac{\tilde{h}(\xi)}{\tilde{h}(\zeta)} = \frac{ \rho h_1(\xi) + (1-\rho)h_2(\xi)}{\rho h_1(\zeta) + (1-\rho)h_2(\zeta)} \geq \frac{\rho M^{-1} + (1-\rho)M^{-1}}{\rho M + (1-\rho)M} = M^{-2}.
		\end{equation*} 
		
		\item[Case 3:] 
		Let $|\xi|_{\vec{a}} > C_2$ and choose $\delta_0 > 0$. By decreasing the value of $\delta_0$, if needed, we can assume that $\delta_0 c_3 \leq \frac{1}{2K}$. Let
		\begin{align*}
		|\xi-\zeta|_{\vec{a}} \leq \delta_0 h_2(\xi) \leq \delta_0 c_3 |\xi|_{\vec{a}}.
		\end{align*}
		Then
		\begin{align*}
		|\xi|_{\vec{a}} &\leq K(|\xi-\zeta|_{\vec{a}}+|\zeta|_{\vec{a}}) \quad \Leftrightarrow \nonumber \\
		|\zeta|_{\vec{a}} &\geq \tfrac{|\xi|_{\vec{a}}}{K} - |\xi-\zeta|_{\vec{a}} \geq ( \tfrac{1}{K} - \delta_0 c_3)|\xi|_{\vec{a}} 
		\geq (\tfrac{1}{K}-\tfrac{1}{2K})|\xi|_{\vec{a}} 
		\geq \tfrac{1}{2K} C_2.
		\end{align*}
		Since $C_2 \geq \frac{8}{3}K$ we have $|\zeta|_{\vec{a}} \geq \frac{4}{3}$. So, the $\d$-moderation of $h_2$ yields
		\begin{equation*}
		R_2^{-1} \leq \frac{\tilde{h}(\xi)}{\tilde{h}(\zeta)} = \frac{h_2(\xi)}{h_2(\zeta)} \leq R_2.
		\end{equation*}
	\end{itemize}
Thus, for $R = \max\{R_1,M,R_2\}$ we conclude that $\tilde{h}$ is d-moderate.
\end{proof}

Given a hybrid regulation function, we can construct a ''nice'' admissible covering of $\mathbb{R}^d\backslash\{0\}$ by open (anisotropic) balls. This will be considered in the following section.

\section{Construction of Structured Admissible Coverings}\label{sec:construction}
In this section we construct structured admissible coverings made up of open (anisotropic) balls in $(\mathbb{R}^d\backslash\{0\},\d, \d \xi)$.
We simplify the construction in the sense that we use a suitable collection of d-balls to cover $\mathbb{R}^d\backslash\{0\}$. Another simplification is that we choose the radius of a given ball in the cover as a suitable function of its center. More specifically, we will use a hybrid regulation function for this purpose. Another perhaps more technically involved approach based on group theory for obtaining structured coverings of $\R^d\backslash\{0\}$ is considered by F\"uhr and Voigtlaender in \cite{MR3345605}.

\begin{mylemma}\label{Lemma5}
Consider $(\mathbb{R}^d\backslash\{0\},\d,\d\xi)$ and let $\tilde{h}: \mathbb{R}^d\backslash\{0\} \to (0,\infty)$ be a hybrid regulation function. 
Pick $0<\delta < \delta_0$. Then
\begin{itemize}
	\item[(1)] there exist an admissible covering $\{B_{d}(\xi_j,\delta \tilde{h}(\xi_j))\}_{j\in J}$ of $\mathbb{R}^d\backslash\{0\}$ and a constant $0<\delta'<\delta$ such that $\{B_{d}(\xi_j,\delta'\tilde{h}(\xi_j))\}_{j\in J}$ are pairwise disjoint.
	\item[(2)] Any two admissible coverings
	\begin{equation*}
\{B_d(\xi_i,\delta_1\tilde{h}(\xi_i))\}_{i\in I} := \{\mathcal{A}_i\}_{i\in I} \quad \text{and} \quad \{B_d(\zeta_j,\delta_1\tilde{h}(\zeta_j))\}_{j\in J} := \{\mathcal{B}_j\}_{j\in J}
	\end{equation*}
	of the type considered in (1) satisfy $\sup_{i\in I} \# J(i) < \infty$ and $\sup_{j\in J} \# I(j) < \infty$, where
\begin{equation*}
J(i) := \{j \; | \; j \in J, \mathcal{A}_i \cap \mathcal{B}_j \ne \emptyset\}, \qquad I(j) := \{i \; | \; i \in I, \mathcal{B}_j \cap \mathcal{A}_i \ne \emptyset\}.
\end{equation*}
\end{itemize}
\end{mylemma}
\begin{proof} 
The proof of Lemma \ref{Lemma5} is a straightforward adaptation of \cite[Lemma 5]{Nielsen}. However, we include the proof of (1) for the sake of completeness.
Fix $0<\delta < \delta_0$ and let $U_\xi := B_d(\xi,\delta \tilde{h}(\xi))$. Define $\nu(\xi) := \{\zeta \in \mathbb{R}^d\backslash\{0\} \; | \; U_\xi \cap U_\zeta \ne \emptyset\}$, and let $U_\xi' = \cup_{\zeta\in \nu(\xi)} U_\zeta$. Suppose $\eta \in U_\xi'$. We would like to estimate $\d(\xi,\eta)$. Notice that $\eta \in U_\zeta$ for some $U_\zeta$ with $U_\xi \cap U_\zeta \ne \emptyset$. Pick a point $\eta_0 \in U_\xi \cap U_\zeta$. Using similar arguments as in \cite[Lemma 5]{Nielsen} we find that 
\begin{align}\label{align:estimate}
\d(\xi,\eta) \leq \delta K \tilde{h}(\xi)(1+2KR^2),
\end{align}
where $K$ is the constant in the quasi triangle inequality for $\d$ and $R$ is the constant associated with $\tilde{h}$. Note that, since $\eta$ was chosen as an arbitrary point in $U_{\xi}'$ and $\tilde{h}$ satisfies $\tilde{h}(\xi) \leq c|\xi|_{\vec{a}}$ for some $c>0$ and all $\xi\in\mathbb{R}^d\backslash\{0\}$ (see Definition \ref{def:hfunction}) it follows that $U_\xi' \subset B_d(\xi,\tilde{R}_1\delta \tilde{h}(\xi)) \subseteq B_d(\xi,\tilde{R}_1\delta c |\xi|_{\vec{a}})$ for $\tilde{R}_1 := K(1+2KR^2)$ and for $\delta > 0$ chosen such that $\tilde{R}_1 \delta c < 1$.
Using Zorn's Lemma, choose a maximal set $\{\xi_j\}_{j\in J} \subset \mathbb{R}^d\backslash\{0\}$ such that $\mathcal{C}_1 := \{V_{\xi_j}\}_{j\in J}$ are pairwise disjoint, with $V_\xi := B_d(\xi,\tfrac{\delta}{\tilde{R}_1}\tilde{h}(\xi))$. By scaling the radii of the balls contained in $\mathcal{C}_1$ by a factor $\tilde{R}_1$ and choosing $\delta > 0$ such that $\tilde{R}_1 \delta c < 1$, it follows by \eqref{align:estimate} that $\mathcal{Q} = \{B_d(\xi_j,\delta \tilde{h}(\xi_j))\}_{j\in J}$ covers $\mathbb{R}^d\backslash\{0\}$.
What remains is to show that for $Q, \# \widetilde{j}$ is uniformly bounded for all $i \in J$. Let $m$ be a constant satisfying $2^m \geq \tilde{R}_1$ and let $\mu$ denote the Lebesgue measure. Fix an $i \in J$. Since $\mu$ satisfies a doubling condition we have
\begin{align*}
\mu\{B_d(\xi_j,\tfrac{\delta}{\tilde{R}_1}\tilde{h}(\xi_j))\} \geq A^{-2m}\mu \{B_d(\xi_j,\tilde{R}_1\delta \tilde{h}(\xi_j))\} \geq A^{-2m}\mu \{U_{\xi_i}\},
\end{align*}
for all $j\in\widetilde{j}$. Now, since $\mathcal{C}_1$ are pairwise disjoint, we have
\begin{equation*}
\# \widetilde{j} \leq \sup_{j\in\widetilde{j}} \frac{\mu\{B_d(\xi_i,\tilde{R}_1\delta \tilde{h}(\xi_i))\}}{\mu\{B_d(\xi_j,\tfrac{\delta}{\tilde{R}_1}\tilde{h}(\xi_j))\}} 
\leq A^{2m} \frac{\mu\{B_d(\xi_i,\tilde{R}_1\delta \tilde{h}(\xi_i))\}}{\mu\{U_{\xi_i}\}} 
\leq A^{3m}.
\end{equation*}
This proves part (1). We refer the reader to \cite[Lemma 5]{Nielsen} for a proof of (2). 
\end{proof}

Lemma \ref{Lemma5} states that, given a hybrid regulation function there exists an admissible covering of $\mathbb{R}^d\backslash\{0\}$. What remains is to show that the coverings are indeed structured. We have the following result. 

\begin{myprop}\label{thm:sdc}
Let the anisotropic norm $|\cdot|_{\vec{a}}:\mathbb{R}^d\backslash\{0\} \to (0,\infty)$ be given as in Definition \ref{def:aninorm} and suppose the quasi-distance $\d$ is induced by $|\cdot|_{\vec{a}}$. Let $\tilde{h}$ be a hybrid regulation function. Then
\begin{itemize}
	\item[(1)] the family $\mathcal{Q}=\{B_{d}(\xi_j,\delta \tilde{h}(\xi_j))\}_{j\in J}$ given in Lemma \ref{Lemma5} is a structured admissible covering of $\mathbb{R}^d\backslash\{0\}$. Moreover, the covering is countable if the topology induced by $\d$ is finer than the Euclidean topology on $\mathbb{R}^d$.
	\item[(2)] Any two such families of structured admissible coverings are equivalent in the sense of \cite[Definition 5]{Nielsen}. That is, $\tilde{h}$ determines exactly one equivalence class of structured admissible coverings.  
\end{itemize}
\end{myprop}

\begin{proof}
Essentially, we can just follow the proof of \cite[Theorem 3]{Nielsen}, but for the sake of convenience we will give the details here. Set $\tau(\cdot) := |\cdot|_{\vec{a}}$, 
let $\delta < \delta_0$, and let $\{\xi_j\}_{j\in J} \subset \mathbb{R}^d\backslash\{0\}$ be the points given in Lemma \ref{Lemma5}. Choose $p_0$ such that the balls $P:= B_d(p_0,1)$ and $Q:=B_d(p_0,2)$ satisfy $P\cap \{0\} = \emptyset$ and $Q\cap \{0\} = \emptyset$. Let
\begin{equation*}
T_j\xi = A_j\xi + (\xi_j - p_0), \quad \text{where} \quad A_j = D_{\vec{a}}(\delta \tilde{h}(\xi_j)).
\end{equation*}
Since $\tau$ satisfies \eqref{eq:conditionaninorm} we have
\begin{equation*}
\tau(T_j^{-1}(\xi-p_0)) = \tau \left( D_{\vec{a}} \left( \frac{1}{\delta \tilde{h}(\xi_j)} \right) ((\xi-p_0)-(\xi_j-p_0))\right) = \frac{1}{\delta \tilde{h}(\xi_j)} \tau(\xi - \xi_j).
\end{equation*}
So
\begin{align*}
P_{j} = \{\xi : \tau(T_j^{-1}(p_0-\xi)) < 1\} = \{\xi : \tau(\xi-\xi_j) < \delta \tilde{h}(\xi_j)\} = B_d(\xi_j,\delta\tilde{h}(\xi_j)).
\end{align*}
Now, using Lemma \ref{Lemma5} it is not difficult to show that $\{P_{j}\}_{j\in J}$ and $\{Q_{j}\}_{j\in J}$ are admissible coverings of $\mathbb{R}^d\backslash\{0\}$. What remains is to verify that \eqref{prop:eq:bounded} holds. Using the d-moderation of $\tilde{h}$ it follows that whenever $B_d(\xi_j,\delta \tilde{h}(\xi_j)) \cap B_d(\xi_i,\delta \tilde{h}(\xi_i)) \ne \emptyset$, then $\tilde{h}(\xi_j) \asymp \tilde{h}(\xi_i)$ uniformly in $j$ and $i$. Thus,
\begin{equation*}
A_j^{-1}A_i = D_{\vec{a}}\left(\delta \tilde{h}(\xi_j)\right)^{-1}D_{\vec{a}}\left(\delta \tilde{h}(\xi_i)\right) = D_{\vec{a}}\left(\delta \tilde{h}(\xi_j)^{-1}\delta\tilde{h}(\xi_i)\right),
\end{equation*}
and \eqref{prop:eq:bounded} holds. To prove countability of $\mathcal{Q}$ when the topology induced by $\d$ is finer than the standard topology, we associate an Euclidean ball with rational radius and center to each set in $\mathcal{Q}$ in such a way that these rational balls are pairwise disjoint. We refer the reader to \cite[Theorem 3]{Nielsen} for a proof of (2). 
\end{proof}

Proposition \ref{thm:sdc} ensures that we can construct a BAPU corresponding to the admissible covering given in Lemma \ref{lemma1}, provided that $\tilde{h}$ is a hybrid regulation function in the sense of Definition \ref{def:hfunction}.

\section{Homogeneous Decomposition Spaces}\label{sec:hmspaces}
Based on coverings of the type considered in Section \ref{sec:construction}, we now give the following definition of homogeneous decomposition spaces. 
Note that, by a $\mathcal{Q}$-moderate weight $w$ on $\mathbb{R}^d\backslash\{0\}$ we mean a strictly positive function satisfying $w(\xi) \leq C w(\zeta)$ for some $C>0$, all $\xi,\zeta \in Q_j$ and all  $j\in J$. 

\begin{mydef}\label{def:complete}
Assume $\mathcal{Q} := \{Q_j\}_{j\in J}$ is a structured admissible covering and let $\{\phi_j^2\}_{j\in J}$ be a corresponding BAPU for $\mathcal{Q}$. Let $\tilde{h}(\xi_j)^\beta, \beta \in \mathbb{R}$ be a $\mathcal{Q}$-moderate weight on $J$ induced by a hybrid regulation function $\tilde{h}$. For $0<p, q\leq \infty$ and $\beta \in \mathbb{R}$ we let $\dot{M}_{p,q}^\beta(\tilde{h})$ consists of those distributions $f\in\mathcal{S}'\backslash\mathcal{P}$ satisfying
\begin{equation*}
	\norm{f}_{\dot{M}_{p,q}^\beta(\tilde{h})} =
	\left( \sum_{j\in J} \left(\tilde{h}(\xi_j)^{\beta} \norm{\phi_j^2(D)f}_{L_p}\right)^q \right)^{1/q} < \infty,
\end{equation*}
with the usual modification if $q=\infty$.
\end{mydef}

In Proposition \ref{prop:complete} we state some fundamental properties of homogeneous decomposition spaces. We will see that the growth condition imposed on the hybrid regulation function from Definition \ref{def:hfunction} is needed to ensure completeness of the spaces. Before presenting the proposition we recall that $\mathcal{S}_0(\mathbb{R}^d)$ is a closed subspace of the Schwartz class $\mathcal{S}(\mathbb{R}^d)$, consisting of all functions $\theta\in\mathcal{S}$ such that
\begin{equation}\label{eq:seminormS0}
\norm{\theta}_M := \sup_{|\beta|\leq M} \sup_{\xi\in\mathbb{R}^d} |\partial^\beta \hat{\theta}(\xi)|(|\xi|^M + |\xi|^{-M}) < \infty \quad \text{for any} \; M\in \mathbb{N}. 
\end{equation}
Moreover, $\mathcal{S}'\backslash\mathcal{P}$ is the set of all continuous linear functionals on $\mathcal{S}_0(\mathbb{R}^d)$. 

\begin{myprop}\label{prop:complete}
	Let $\tilde{h}$ be a hybrid regulation function. For $0<p, q\leq \infty$ and $\beta \in \mathbb{R}$ we have
	\begin{enumerate}
		\item The continuous embeddings
		\begin{equation}\label{eq:embeds} \mathcal{S}_0 \hookrightarrow \dot{M}_{p,q}^\beta(\tilde{h}) \hookrightarrow \mathcal{S}'\backslash \mathcal{P}. \end{equation}
		\item ${\dot{M}_{p,q}^\beta}(\tilde{h})$ is a quasi-Banach space (Banach space if $1\leq p,q$).
		\item If $0<p<\infty$ and $0<q<\infty$ then the space $\mathcal{S}_0(\mathbb{R}^d)$ is dense in $\dot{M}_{p,q}^\beta(\tilde{h})$.
	\end{enumerate}
\end{myprop}
The proof of Proposition \ref{prop:complete} can be found in Appendix \ref{sec:appB}. 

\section{Distributional convergence}\label{sec:conv}
In this section we consider expansions of any tempered distribution $f$ in terms of the system from Proposition \ref{prop1} (ii). 
We begin by defining an ordering on the index set $J$. Since $J$ is countable we can assume $J = \mathbb{Z}$. We define an ordering on $J=\mathbb{Z}$ by associating $\mathbb{Z}^+$ with $J_2$, where
\begin{equation*}
J_2 := \left\{j \in J : B_d(\xi_j,\delta \tilde{h}(\xi_j)) \cap B_d\left(0,C\right)=\emptyset\right\},
\end{equation*}
with the value of $C$ specified in the proof of Lemma \ref{lemma1}. Put  $J_1 = J\backslash J_2$ and associate $\mathbb{Z}^-$ with $J_1$.
We shall in the rest of this paper assume that, for an affine transformation $T_j = A_j\cdot + b_j$, the matrix $A_j$ is given as $A_j := D_{\vec{a}}(\delta \tilde{h}(\cdot))$, where $D_{\vec{a}}(\cdot)$ is defined in Section \ref{sec:background}. \\
 
The following lemma provides a universal decomposition of tempered distributions relative to the system from Proposition \ref{prop1} (ii). 

\begin{mylemma}\label{lemma:konvergens}
Let $\mathcal{Q} := \{B_d(\xi_j,\delta\tilde{h}(\xi_j))\}_{j\in J}$ be a structured admissible covering of $\mathbb{R}^d\backslash\{0\}$ and let the system $\{\phi_j\}_{j\in J} \subset \mathcal{S}(\mathbb{R}^d)$ be defined as in Proposition \ref{prop1} (ii).
For any tempered distribution $f\in \mathcal{S}'(\mathbb{R}^d)$ we have
\begin{equation}\label{eq:series}
f = \sum_{j\in J_1} \check{\phi}^2_j * f + \sum_{j\in J_2} \check{\phi}^2_j * f,
\end{equation}
where the convergence of the series are in $\mathcal{S}'/\mathcal{P}$. 
To be more precise, this means there exists a constant $p$ depending only on the order of the distribution $\hat{f}$, a sequence of polynomials $\{P_k\}_{k=1}^\infty \subset \mathcal{P}$ with $\deg P_k \leq p$, and $P\in\mathcal{P}$, such that
\begin{equation*}
f = \lim_{k,K\to \infty} \left ( \sum_{j=-k}^K \check{\phi}_j^2 * f + P_k \right) + P,
\end{equation*}
where the convergence of the series is in $\mathcal{S}'$. 
\end{mylemma}

\begin{myre}
The proof of Lemma \ref{lemma:konvergens} given below will show that one has a similar convergence result for BAPUs associated with the inhomogeneous decomposition spaces  considered in \cite[Proposition 1]{Nielsen}.
\end{myre}
To prove Lemma \ref{lemma:konvergens} we will need the following proposition, which was stated by Peetre in \cite{MR0461123} in a related context. A proof of the result can be found in \cite[Section 6]{BH05}.

\begin{myprop}\label{prop:distrib}
Suppose $\{f_i\}_{i\in\mathbb{N}}$ is a sequence of distributions in $\mathcal{S}'(\mathbb{R}^d)$ and $p\geq 0$ is an integer. Assume that for every multi-index $\gamma$ with $|\gamma| = p+1$, the sequence of partial derivatives $\{\partial^\gamma f_i\}$ converges in $\mathcal{S}'$ as $i\to \infty$. Then there exists a sequence of polynomials $\{P_i\}_{i\in\mathbb{N}}$ with $\deg P_i \leq p$ such that $\{f_i + P_i\}$ converges to some distributions $f\in\mathcal{S}'$ as $i\to\infty$. 
\end{myprop}

\begin{proof} [Proof of Lemma \ref{lemma:konvergens}]
Take an arbitrary $f\in \mathcal{S}'$. Then  $\hat{f} \in \mathcal{S}'$. Hence, $\hat{f}$ is a bounded functional with respect to the semi-norm on $\mathcal{S}(\mathbb{R}^d)$ and is therefore of finite order. We assume $\hat{f}$ has order $\leq m$. More precisely, there exists an integer $l \geq 0$ and a constant $C$ such that
\begin{equation}\label{eq:finiteorder}
\left|\ip{\hat{f},\psi}\right| \leq C \sup_{|\alpha|\leq l, |\beta| \leq m} \norm{\psi}_{\alpha,\beta}, \quad \text{for all} \; \psi \in \mathcal{S},
\end{equation}
where $\norm{\psi}_{\alpha,\beta} := \sup_{x\in\mathbb{R}^d} |x^\alpha| |\partial^\beta \psi(x)|$ for every multi-indices $\alpha,\beta \in \mathbb{N}_0^d$. \\

To show convergence of the series in \eqref{eq:series} we consider two cases.
\begin{itemize}
	\item[Case 1:]
Let $j\in J_2$. The support of $\phi_j$ satisfies

\begin{equation}\label{eq:support}
\text{supp}(\phi_j) \subset \{\xi \in B_d(\xi_j,\delta \tilde{h}(\xi_j)) \; : \; |\xi-\xi_{j}|_{\vec{a}} \leq \delta \tilde{h}(\xi_{j})\}.
\end{equation}
To show that the series $\sum_{j\in J_2} \check{\phi}^2_j * f$ converges in $\mathcal{S}'$, we use that the Fourier transform is an isomorphism of $\mathcal{S}'$. This allows us to consider convergence of the series $\sum_{j \in J_2} \phi^2_j  \hat{f}$ in $\mathcal{S}'$. 
Since $\phi^2_j \psi \in \mathcal{S}$ we have by \eqref{eq:finiteorder},
\begin{equation}\label{eq:norm}
\left|\ip{\phi^2_j \hat{f},\psi}\right| = \left|\ip{\hat{f}, \phi^2_j \psi}\right| \leq C \sup_{|\alpha|\leq l, |\beta| \leq m} \norm{\phi^2_j \psi}_{\alpha,\beta}.
\end{equation}
We consider $\partial^\beta \phi_j^2$, where $\phi_j$ is defined in \eqref{eq:squarerootbapu}.
Let $u(\xi) := \Phi^2(A_j^{-1}\xi - A_j^{-1}b_j)$, $v(\xi) := \sum_{k\in J_2} \Phi^2(A_k^{-1}\xi - A_k^{-1}b_k)$ and recall the formula in \eqref{eq:Leib}. 
By the same arguments as in the proof of Proposition \ref{prop1}, it suffices to consider the numerator of the expression on the right hand side of \eqref{eq:Leib}. 
Applying the chain rule shows that $\sup_{|\lambda|=s_1} |\partial^\lambda u| \leq C_\lambda \norm{A_j^{-1}}^{s_1} \sup_{|\lambda|=s_1}|\partial^\lambda \Phi^2|$. Since a similar estimate holds for $\sup_{|\beta-\lambda|=s_2}|\partial^{\beta-\lambda} v|$ we have
\begin{equation*}
\sup_{|\beta|=s_1+s_2 = s} |\partial^\beta \phi_j^2| \leq C \norm{A_{j}^{-1}}^s \sup_{|\beta|=s}|\partial^\beta \Phi^2|.
\end{equation*}
Since $j\in J_2$, there exists a constant $M<\infty$ such that $M^{-1} \leq \tilde{h}(\xi_j)$ for all $j\in J_2$ (see the proof of Lemma \ref{lemma1}).
This implies that each entry in $A_j^{-1}$ is uniformly bounded, thus the norms $\norm{A_{j}^{-1}}$ are uniformly bounded for all $j\in J_2$. Using this together with Leibniz's formula and \eqref{eq:support} yields
\begin{flalign*}
\sup_{|\alpha|\leq l, |\beta| \leq m} \norm{\phi^2_j  \psi}_{\alpha,\beta}
&\leq \sup_{\xi \in \mathbb{R}^d} (1+|\xi|)^l \; \sup_{|\beta| \leq m}|\partial^\beta \phi^2_j (\xi) \psi(\xi)| \\
&\leq C  \sup_{\xi \in \mathbb{R}^d} \left( (1+|\xi|)^l \sup_{|\beta|\leq m} |\partial^\beta \phi^2_j (\xi)| \cdot \sup_{|\beta|\leq m} |\partial^{\beta} \psi(\xi)|\right) \\
&\leq C \sup_{\xi \in B_d(\xi_j,\delta \tilde{h}(\xi_j))} (1+|\xi|)^l \sup_{|\beta|\leq m} |\partial^\beta \psi(\xi)| \\
&\leq C \sup_{|\alpha|\leq l+d+1, |\beta| \leq m} \norm{\psi}_{\alpha,\beta} \; \sup_{\xi\in B_d(\xi_j,\delta \tilde{h}(\xi_j))} (1+|\xi|)^{-d-1}.
\end{flalign*}
Now we need to sum over $J_2$. But for $\xi \in B_d(\xi_j,\delta \tilde{h}(\xi_j))$ we have $\sum_{j\in J_2} (1+|\xi|)^{-d-1} < \infty$ since the covering $\mathcal{Q}$ has finite height and $\tilde{h}(\xi_j) \geq M^{-1}$ for $j\in J_2$.
Combining this with \eqref{eq:norm} shows that the series $\sum_{j\in J_2} \phi_j^2 \hat{f}$ converges in $\mathcal{S}'$.\\

\item[Case 2:] Let $j\in J_1$ and note by definition of $J_1$ that $\tilde{h}(\xi_j) = h_1(\xi_j)$ for all $j\in J_1$. Thus the support of $\phi_j$ satisfies
\begin{equation}\label{eq:supporth1}
\text{supp}(\phi_j) \subset \{\xi \in B_d\left(\xi_j,\delta h_1(\xi_j)\right) \; : \; |\xi-\xi_j|_{\vec{a}} \leq \delta h_1(\xi_j)\}.
\end{equation}
To use the result of Proposition \ref{prop:distrib} we must show that for sufficiently large $p$, the series $\sum_{j\in J_1} \partial^\gamma (\check{\phi}_j^2 * f)$ converges in $\mathcal{S}'$ for every multi-index $|\gamma| = p+1$. Again, since the Fourier transform is an isomorphism of $\mathcal{S}'$, this is equivalent to proving that the series $\sum_{j\in J_1} \xi^\gamma \phi_j^2 \hat{f}$ converges in $\mathcal{S}'$ with $|\gamma| = p+1$. 
Recall that $h_1$ satisfies $h_1(\xi) \geq c_0|\xi|_{\vec{a}}^r$ for some $c_0>0$ and $r \geq 1$ (see Definition \ref{def:hfunction}). Using the formula in \eqref{eq:Leib} we have
\begin{align}\label{eq:h1estimate}
\sup_{|\beta|=s} \left| \partial^\beta \phi_j^2(\xi) \right| &\leq C \norm{A_j^{-1}}^s \sup_{|\beta|=s} |\partial^\beta \Phi^2|
\leq C|\xi_j|_{\vec{a}}^{-rd\cdot s} \sup_{|\beta|=s} |\partial^\beta \Phi^2|,
\end{align}
where $d = \sum_{i=1}^d a_i$. Choose any integer $p > \alpha_2(rd+rdm) + m - 1$, where $\alpha_2$ is the value occurring in \eqref{eq:anormJ1}. Using \eqref{eq:h1estimate} together with \eqref{eq:supporth1} we have

\allowdisplaybreaks
\begin{flalign}\label{eq:smaafrekvens}
&\sup_{|\alpha|\leq l, |\beta| \leq m} \norm{\xi^\gamma \phi_j^2 \psi}_{\alpha,\beta} \nonumber \\
&\leq \sup_{\xi \in \mathbb{R}^d} (1+|\xi|)^l \sup_{|\beta|\leq m} \left| \partial^\beta \xi^\gamma \phi_j^2(\xi) \psi(\xi) \right| \nonumber \\
&\leq C \sup_{0\leq k \leq m} \sup_{\xi \in \mathbb{R}^d} \left ( (1+|\xi|)^l |\xi|^{p+1-k} \sup_{|\beta|\leq m-k} |\partial^\beta \phi_j^2(\xi)| \cdot \sup_{|\beta|\leq m-k} |\partial^\beta \psi(\xi)| \right) \nonumber \\
&\leq C \sup_{0\leq k \leq m} |\xi_j|_{\vec{a}}^{-rd(m-k)} \sup_{|\alpha|\leq l, |\beta|\leq m} \norm{\psi}_{\alpha,\beta} \sup_{\xi \in B_d(\xi_j,\delta h_1(\xi_j))} |\xi|^{p+1-k} \nonumber \\
&\leq C \sup_{0\leq k \leq m} |\xi_j|_{\vec{a}}^{-rd(m-k)+1/\alpha_2(p+1-k)} \sup_{|\alpha|\leq l, |\beta|\leq m} \norm{\psi}_{\alpha,\beta},
\end{flalign}
where we in the last step used \eqref{eq:anormJ1} and the fact that $h_1(\xi)$ is d-moderate. Now we need to sum over $J_1$. Note that, by construction of the covering, $|B_d(0,C)| \asymp \sum_{j \in J_1} h_1(\xi_j)^d$. So $\sum_{j\in J_1} |\xi_j|_{\vec{a}}^{rd} \leq c |B_d(0,C)| < \infty$. 
Therefore, by \eqref{eq:norm} and \eqref{eq:smaafrekvens} we have, for any $|\gamma| = p+1$ and by our choice of $p$ that
\begin{equation*}
\left|\ip{\xi^\gamma \phi_j^2\hat{f},\psi}\right| = \left|\ip{\hat{f},\xi^\gamma \phi_j^2\psi}\right| \leq C \sup_{|\alpha|\leq l, |\beta|\leq m} \norm{\xi^\gamma \phi_j^2\psi}_{\alpha,\beta} < \infty.
\end{equation*}
Because $|\gamma| = p+1$ is arbitrary we conclude by Proposition \ref{prop:distrib} that there exists a sequence of polynomials $\{P_k\}$ such that $\{\sum_{j=-k}^{-1} \check{\phi}_j^2 * f + P_k\}$ converges in $\mathcal{S}'$ as $k\to \infty$.
\end{itemize}
Combining the above shows that $\{\sum_{j=-k}^K \check{\phi}_j^2 * f + P_k\}$ converges to some distribution $f_0 \in \mathcal{S}'$ as $k,K\to\infty$. Using that  $\sum_{j\in J} \phi_j^2(\xi) = 1$ it follows that $\text{supp}(\hat{f}-\hat{f}_0) = \{0\}$ by testing against $\psi \in \mathcal{S}$ with $0 \notin \text{supp}(\psi)$. Hence we conclude that there exists a polynomial $P$ such that $f=f_0 + P$. This completes the proof.  
\end{proof}

In the following section we will construct a tight frame system for $L_2(\mathbb{R}^d)$ and use the result of Lemma \ref{lemma:konvergens} to show that any function or tempered distribution has an expansion in terms of the aforementioned system.

\section{Tight Frames} \label{section:frames}
Here we construct tight frames for $L_2(\mathbb{R}^d)$ adapted to a given structured admissible covering based on the system from Proposition \ref{prop1} (ii). The method used below to construct tight frames for $L_2(\mathbb{R}^d)$ was first introduced in \cite{Nielsen}.

Assume $\mathcal{Q} = \{Q_j\}_{j\in J}$ is a structured admissible covering of $\mathbb{R}^d\backslash\{0\}$. Let $K_a$ be a cube in $\mathbb{R}^d$ (aligned with the coordinate axes) with side-length $2a$ satisfying $Q \subseteq K_a$. We define
\begin{equation}\label{eq:ortbasis}
e_{n,j}(\xi) := (2a)^{-\tfrac{d}{2}}|T_j|^{-\tfrac{1}{2}} \chi_{K_a}(T_j^{-1}\xi)\e^{i\tfrac{\pi}{a}n\cdot T_j^{-1}\xi}, \quad n\in\mathbb{Z}^d, j \in J,
\end{equation}
and
\begin{equation*}
\hat{\eta}_{n,j} := \phi_j e_{n,j}, \quad n\in \mathbb{Z}^d, j \in J,
\end{equation*}
with $\phi_j$ given in Proposition \ref{prop1} (ii). We can also obtain an explicit representation of $\eta_{n,j}$ in direct space. Let $T_j = A_j\cdot + b_j$ and define $\hat{\mu}_{j}(\xi) := \phi_{j}(T_j\xi)$. Then, by a simple substitution we find that
\begin{align}\label{eq:eta}
\eta_{n,j}(x) = (2a)^{-d/2} |T_j|^{1/2} \e^{ix\cdot b_j} \mu_j(\tfrac{\pi}{a}n + A_j^Tx).
\end{align}
Since $\phi_j \in \mathcal{S}(\mathbb{R}^d)$ has compact support in $Q_j$, all of its partial derivatives are continuous and have compact support. Therefore $ |\partial^{\beta} \hat{\mu}_{j}(\xi)| = |\partial^{\beta} \phi_j(T_j\xi)| \leq C_{\beta}\chi_{Q_j}(\xi)$ for every $\beta \in \mathbb{N}^d$. Thus, we have
\begin{align*}
|\mu_{j}(x)| 
= \left|\mathcal{F}^{-1}\phi_j(T_j x)\right| \leq C |(ix)^{-|\beta|}| \; \int_{\mathbb{R}^d} \left|\partial^\beta \hat{\mu}_j(\xi)\right| \d \xi.
\end{align*}
Now, for any $N\in\mathbb{N}$ we sum over $|\beta| \leq N$ and use that $\sum_{|\beta|\leq N}\left|x^\beta\right| \asymp  (1+|x|)^N$ to obtain
\begin{align}\label{eq:estimate}
|\mu_j(x)| &\leq C(1+|x|)^{-N} \sum_{|\beta|\leq N} \norm{\partial^\beta \hat{\mu}_j}_{L_1} \leq C_N(1+|x|)^{-N},
\end{align}
with $C_N$ a constant that is independent of $j\in J$.
We next verify that the system $\{\eta_{n,j}\}_{n\in\mathbb{Z}^d,j\in J}$ constitutes a tight frame for $L_2(\mathbb{R}^d)$. Note that, since $\{e_{n,j}\}_{n\in\mathbb{Z}^d, j\in J}$ forms an orthonormal basis for $L_2(T_j(K_a))$ and $\text{supp}(\phi_j) \subset T_j(K_a)$ we have
\begin{align*}
\sum_{n\in\mathbb{Z}^d} |\ip{f,\eta_{n,j}}|^2 &= \sum_{n\in\mathbb{Z}^d} \left| \ip{\phi_j \hat{f},e_{n,j}}\right|^2 = \norm{\phi_j \hat{f}}^2_{L_2}.
\end{align*}
Using that $\{\phi_j^2\}_{j\in J}$ is a partition of unity yields
\begin{equation*}
\sum_{j\in J} \sum_{n\in\mathbb{Z}^d} |\ip{f,\eta_{n,j}}|^2 = \sum_{j\in J} \norm{\phi_j \hat{f}}^2_{L_2} = 
\int_{\mathbb{R}^d} \sum_{j\in J} \phi_j^2(\xi)|\hat{f}(\xi)|^2 \d \xi = \norm{f}^2_{L_2}.
\end{equation*}
We can also obtain the frame expansion of $\hat{f}$ directly. Note that
\begin{equation*}
\phi_j\hat{f} = \sum_{n\in\mathbb{Z}^d} \ip{\phi_j\hat{f},e_{n,j}}e_{n,j} = \sum_{n\in\mathbb{Z}^d} \ip{\hat{f},\hat{\eta}_{n,j}}e_{n,j},
\end{equation*}
so
\begin{equation} \label{eq:phi}
\phi_j^2\hat{f} = \sum_{n\in\mathbb{Z}^d} \ip{\hat{f},\hat{\eta}_{n,j}}\phi_je_{n,j} = \sum_{n\in\mathbb{Z}^d} \ip{\hat{f},\hat{\eta}_{n,j}}\hat{\eta}_{n,j}.
\end{equation}
Again, since $\{\phi_j^2\}_{j\in J}$ is a partition of unity we have
\begin{equation*}
\hat{f} = \sum_{j\in J} \phi_j^2 \hat{f} = \sum_{j\in J} \sum_{n\in\mathbb{Z}^d} \ip{\hat{f},\hat{\eta}_{n,j}}\hat{\eta}_{n,j}.
\end{equation*}

\begin{myre}
The above arguments shows that we obtain a tight frame for $L_2(\mathbb{R}^d)$ regardless of the choice of cube $K_a$ as long as $Q \subseteq K_a$, and all we need is a partition of unity to be able to construct frames for $L_2(\mathbb{R}^d)$. 
\end{myre}

Using the tight frame $\{\eta_{n,j}\}_{n\in\mathbb{Z}^d,j\in J}$ we state the following result.

\begin{myprop}\label{prop:reproduce}
	Suppose $K_a$ is a cube in $\mathbb{R}^d$ (aligned with the coordinate axes) with side-length $2a$ satisfying $Q \subseteq K_a$. If
	$g\in \mathcal{S}'(\mathbb{R}^d), h\in \mathcal{S}(\mathbb{R}^d)$ and 
	\begin{equation*}
	\text{supp}(\hat{g}), \; \text{supp}(\hat{h}) \subset T_j(K_a), \quad \text{for some} \; j\in J,
	\end{equation*}
	then
	\begin{equation}\label{eq:gfoldetmedh}
	(g*h)(x) = \sum_{n\in\mathbb{Z}^d} (2a)^{-d} |T_j|^{-1} g(-T_j^{-1}\tfrac{\pi}{a}n)h(T_j^{-1}\tfrac{\pi}{a}n-x),
	\end{equation}
	where the series converges in $\mathcal{S}'$.
	Consequently, if $\phi_j$ satisfy \eqref{eq:phisatisfy1}, \eqref{eq:phisatisfy2}, then, for any $f \in \mathcal{S}'\backslash\mathcal{P}$ we have
	
	\begin{equation}\label{eq:fexpansion}
	f=\sum_{j\in J} \sum_{n\in\mathbb{Z}^d} \ip{f,\eta_{n,j}}\eta_{n,j},
	\end{equation}
	where the convergence of the series is in $\mathcal{S}'\backslash\mathcal{P}$. To be more precise, this means there exists a sequence of polynomials $\{P_k\}_{k=1}^\infty \subset \mathcal{P}$ and $P\in \mathcal{P}$ such that
	\begin{equation*}
		f = \lim\limits_{\substack{k,K,N \to \infty}} \left( \sum_{j= -k}^{K} \; \sum_{n=-N}^N \ip{f,\eta_{n,j}}\eta_{n,j} + P_k\right) + P,
	\end{equation*}
	with convergence in $\mathcal{S}'$.
\end{myprop}

\begin{proof}
The proof of Proposition \ref{prop:reproduce} is an adaptation of \cite[Lemma 2.8]{BH05}. We begin by considering $g\in\mathcal{S}(\mathbb{R}^d)$. Since the functions in \eqref{eq:ortbasis} constitute an orthonormal basis for $L_2(T_j(K_a))$, we can expand $\hat{g}$ in this basis,
	\begin{align}\label{eq:basis}
\hat{g}(\xi) = \sum_{n\in\mathbb{Z}^d} (2a)^{-d} |T_j|^{-1} \left(\int_{T_j(K_a)} \hat{g}(\zeta) \e^{i(-T_j^{-1}\tfrac{\pi}{a}n) \cdot \zeta} \d \zeta \right) \e^{iT_j^{-1}\tfrac{\pi}{a}n\cdot \xi}.
	\end{align}
Since $\hat{g}$ has compact support in $T_j(K_a)$ we may integrate over $\mathbb{R}^d$ in the above integral without affecting the expression. Hence, with the use of Fourier's inversion formula we find that 
\begin{align*}
\hat{g}(\xi) &= \sum_{n\in\mathbb{Z}^d} (2a)^{-d} |T_j|^{-1} (2\pi)^{d/2} g(-T_j^{-1} \tfrac{\pi}{a}n) \e^{iT_j^{-1}\tfrac{\pi}{a}n\cdot \xi}, \quad \text{for} \; \xi \in T_j(K_a).
	\end{align*}
	Since $\text{supp}(\hat{h}) \subset T_j(K_a)$, we may replace $\hat{g}$ by its periodic extension without altering the product $\hat{g}\hat{h}$. Using that $g*h = (\hat{g}\hat{h})^{\check{}}$ we obtain
	\begin{align}\label{align:sum}
		(g*h)(x) &= \sum_{n\in\mathbb{Z}^d} (2a)^{-d} |T_j|^{-1} (2\pi)^{d/2} g(-T_j^{-1}\tfrac{\pi}{a}n)\left( \e^{iT_j^{-1}\tfrac{\pi}{a}n\cdot \xi} \hat{h}(\xi)\right)^{\check{}}(x) \nonumber \\
		&= \sum_{n\in\mathbb{Z}^d} (2a)^{-d} |T_j|^{-1} g(-T_j^{-1}\tfrac{\pi}{a}n) h\left(T_j^{-1}\tfrac{\pi}{a}n-x \right).
	\end{align}
Using the expansion \eqref{align:sum} we now consider the case of a general distribution $g\in \mathcal{S}'$. Since $\hat{g} \in \mathcal{S}'(\mathbb{R}^d)$ has compact support in the Fourier domain, $g$ is a regular distribution, i.e. $g$ is at most of polynomial growth. Let $\delta > 0$ and set $g_\delta (x) = \gamma (\delta x)g(x)$, where $\gamma \in \mathcal{S}(\mathbb{R}^d)$ satisfies $\gamma(0)=1$, and $\text{supp}(\hat{\gamma})$ is compact. Then $g_\delta(x) \in \mathcal{S}(\mathbb{R}^d)$, so by \eqref{align:sum} we have
	\begin{equation*}
		(g_\delta * h)(x) = \sum_{n\in\mathbb{Z}^d} (2a)^{-d} |T_j|^{-1} g_\delta (-T_j^{-1}\tfrac{\pi}{a}n)h(T_j^{-1}\tfrac{\pi}{a}n-x).
	\end{equation*}
	Because of the growth behaviour of $g$, we may use Lebesgue's dominated convergence theorem. Thus, taking the limit as $\delta \to 0$ we obtain \eqref{eq:gfoldetmedh} with pointwise convergence in $\mathcal{S}'$.
	
Now, to show \eqref{eq:fexpansion} take any $j\in J$, let $\hat{g} = \hat{f}\phi_j$ and $\hat{h} = \phi_j$. By \eqref{eq:gfoldetmedh}, or as seen more directly by \eqref{eq:basis}, we have
	\begin{equation*}
		f*\check{\phi}_j * \check{\phi}_j = (\hat{f}\phi^2_j)^{\check{}} = \sum_{n\in\mathbb{Z}^d} \ip{f,\eta_{n,j}}\eta_{n,j}, \quad \text{for} \; f\in\mathcal{S}'.
	\end{equation*}
By summing the above over $j\in J$ and using Lemma \ref{lemma:konvergens} together with \eqref{eq:phisatisfy1} and \eqref{eq:phisatisfy2} we obtain \eqref{eq:fexpansion}. Hence the proof is completed.
\end{proof}
	
\begin{myre}
It follows directly from the proof of Proposition  \ref{prop:reproduce} that one has a similar distributional convergence result for the frames  considered in \cite{Nielsen} associated with inhomogeneous decomposition spaces. In fact, in the inhomogeneous case, there is no need to add polynomials to induce convergence and one will have convergence of the corresponding expansion \eqref{eq:fexpansion} in the sense of tempered distributions.  
\end{myre}	

\section{Characterization of \texorpdfstring{$\dot{M}_{p,q}^\beta(\tilde{h})$}--spaces} \label{sec:char}
The goal of this section is to show that the tight frame $\{\eta_{n,j}\}_{n\in\mathbb{Z}^d,j\in J}$ from Section \ref{section:frames} gives an atomic decomposition of the homogeneous decomposition spaces. By this we mean that the canonical coefficient operator is bounded on $\dot{M}_{p,q}^\beta(\tilde{h})$ into a suitable coefficient space on which there is a bounded reconstruction operator. 
We begin with the following lemma, which later will be used
to show that it is possible to express the $\dot{M}_{p,q}^\beta(\tilde{h})$-norm by the canonical frame coefficients.

\begin{mylemma} \label{lemma:stability}
Given $f\in  \mathcal{S}_0(\mathbb{R}^d)$. For $0<p\leq\infty$ there exists $C_p, C_p' < \infty$, such that
\begin{align*}
\left( \sum_{n\in\mathbb{Z}^d} |\ip{f,\eta_{n,j}}|^p\right)^{1/p} &\leq C_p|T_j|^{1/p-1/2} \norm{\widetilde{\phi}^2_j(D)f}_{L_p}, \quad \text{and}  \\
\norm{\phi^2_j(D)f}_{L_p} &\leq C_p'|T_j|^{1/2-1/p} \left( \sum_{n\in\mathbb{Z}^d} |\ip{f,\eta_{n,j}}|^p\right)^{1/p},
\end{align*}
with norm equivalence independent of $j\in J$. When $p=\infty$ the sum over $n\in\mathbb{Z}^d$ is changed to $\sup$. 
\end{mylemma}

\begin{proof}
The proof of Lemma \ref{lemma:stability} is an adaptation of \cite[Lemma 2]{Nielsen}, which is included for the sake of completeness. 
Using the estimate derived in \eqref{eq:estimate} together with the representation \eqref{eq:eta} we have
\begin{align}\label{eq:etaestimates1}
&\sup_{x\in\mathbb{R}^d} \norm{\{\eta_{n,j}(x)\}_{n\in\mathbb{Z}^d}}_{\ell_p} \leq C_p|T_j|^{1/2} \quad \text{and} \\ 
&\label{eq:etaestimates2}\sup_{n\in\mathbb{Z}^d} \norm{\eta_{n,j}}_{L_p} \leq C'_p|T_j|^{1/2-1/p}.
\end{align}
Suppose $p\leq 1$. 
Since $\text{supp}(\hat{\eta}_{n,j}) \subset T_j(K_a)$ and $\text{supp}(\widetilde{\phi}^2_j) \subset \cup_{k\in \widetilde{j}} T_{k}(K_a) \subset T_j(K_{a'})$, where $K_{a'}$ is a cube with side-length $2a'$ we have $\text{supp}((\widetilde{\phi}^2_j * \hat{\eta}_{n,j})(T_j\cdot)) \subset K_{2a}$, a cube with side-length $4a$. 
With $\hat{f}(\xi) = \left( \widetilde{\phi}^2_j * \hat{\eta}_{n,j}\right)(T_j(\xi))$ we use Lemma \ref{lemma:nikolskij} to obtain
\begin{align*}
\sum_{n\in\mathbb{Z}^d} |\ip{f,\eta_{n,j}}|^p &= \sum_{n\in\mathbb{Z}^d} \left| \ip{\widetilde{\phi}^2_j(D)f,\eta_{n,j}}\right|^p \leq \sum_{n\in\mathbb{Z}^d} \norm{(\widetilde{\phi}^2_j(D)f)\eta_{n,j}}^p_{L_1} \\
&\leq C_p|T_j|^{1-p} \sum_{n\in\mathbb{Z}^d} \norm{(\widetilde{\phi}^2_j(D)f)\eta_{n,j}}^p_{L_p} \leq C_p|T_j|^{1-\tfrac{p}{2}}\norm{\widetilde{\phi}^2_j(D)f}^p_{L_p},
\end{align*}
where we used \eqref{eq:etaestimates1} in the final step. This proves the first inequality in the lemma for $p\leq 1$.  Now, to prove the other estimate suppose first $p\leq 1$. Using the expression in \eqref{eq:phi} we have
\begin{align*}
\norm{\phi_{j}^2(D)f}_{L_p}^p &= \norm{\mathcal{F}^{-1}\left( \sum_{n \in\mathbb{Z}^d} \ip{\hat{f},\hat{\eta}_{n,j}}\hat{\eta}_{n,j}\right)}_{L_p}^p 
\leq \sum_{n\in\mathbb{Z}^d}|\ip{f,\eta_{n,j}} |^p\norm{\eta_{n,j}}_{L_p}^p \nonumber \\
&\leq C_p'|T_j|^{\tfrac{p}{2}-1} \sum_{n\in\mathbb{Z}^d} |\ip{f,\eta_{n,j}}|^p. \nonumber
\end{align*}
This proves the lemma for $p\leq 1$. 
For $1<p<\infty$ the estimates follows using \eqref{eq:etaestimates1} and \eqref{eq:etaestimates2} for $p=1$ together with Hölder's inequality. We leave the case $p=\infty$ for the reader.
\end{proof}

We now consider a characterization of the homogeneous decomposition spaces $\dot{M}_{p,q}^\beta(\tilde{h})$ in terms of the canonical frame coefficients. For notational convenience let $\eta_{n,j}^p := |T_j|^{1/2-1/p} \eta_{n,j}$.
Then $\eta_{n,j}^p$ denotes the function $\eta_{n,j}$ ''normalized'' in $L_p(\mathbb{R}^d)$ for $p\in (0,\infty]$. We now combine Proposition \ref{prop:reproduce}, Lemma \ref{lemma:stability}, and the fact that $\mathcal{S}_0(\mathbb{R}^d)$ is dense in $\dot{M}_{p,q}^\beta(\tilde{h})$ (see Proposition \ref{prop:complete})
to conclude that the following result holds.

\begin{myprop}\label{prop:framecoeff}
Every tempered distribution modulo polynomials, $f\in \mathcal{S}'\backslash\mathcal{P}$ has an expansion of the form
\begin{equation}\label{eq:expanddistribution}
f = \sum_{j\in J} \sum_{n\in\mathbb{Z}^d} \ip{f,\eta^p_{n,j}}\eta^p_{n,j},
\end{equation}
with convergence in $\mathcal{S}'\backslash\mathcal{P}$. 
Moreover, given $0<p, q\leq \infty$ and $\beta\in \mathbb{R}$ we have the characterization
\begin{equation}\label{eq:equivnorm}
\norm{f}_{\dot{M}_{p,q}^\beta(\tilde{h})} \asymp \norm{\left\{ \tilde{h}(\xi_j)^\beta \left( \sum_{n\in\Z^d} \left|\ip{f,\eta_{n,j}^p}\right|^p \right)^{1/p} \right\}_{j\in J}}_{\ell_q},
\end{equation}
where the constants of the equivalence depend on $p$ and $q$. For $p=\infty$ the sum over $n\in\mathbb{Z}^d$ is changed to $\sup$.
\end{myprop}

Inspired by Proposition \ref{prop:framecoeff} let us define the following coefficient space.
\begin{mydef}
Given $0<p, q\leq \infty$ and $\beta \in \mathbb{R}$, we define the space $\dot{m}_{p,q}^\beta(\tilde{h})$ as the set of coefficients $c=\{c_{n,j}\}_{n\in\mathbb{Z}^d,j\in J} \subset \mathbb{C}$ satisfying
\begin{align*}
\norm{c}_{\dot{m}_{p,q}^\beta(\tilde{h})} &:= \norm{\left\{ \tilde{h}(\xi_j)^{\beta + \frac{d}{2}-\frac{d}{p}} \left(\sum_{n\in\Z^d}|c_{n,j}|^p\right)^{1/p}\right\}_{j\in J}}_{\ell_q} < \infty
\end{align*}
\end{mydef}

Note that the characterization \eqref{eq:equivnorm} in Proposition \ref{prop:framecoeff} ensures that the canonical coefficient operator $C : \dot{M}_{p,q}^\beta(\tilde{h}) \to \dot{m}_{p,q}^\beta(\tilde{h})$ is bounded and that the norms in $\dot{M}_{p,q}^\beta(\tilde{h})$ and $\dot{m}_{p,q}^\beta(\tilde{h})$ are equivalent. Now, if we can verify the existence of a bounded reconstruction operator $R: \dot{m}_{p,q}^\beta(\tilde{h}) \to \dot{M}_{p,q}^\beta(\tilde{h})$ we have shown that the system $\{\eta_{n,j}\}_{n\in\mathbb{Z}^d,j\in J}$ gives an atomic decomposition of the homogeneous decomposition spaces. The following lemma gives a way to define such an operator.

\begin{mylemma}\label{lemma:reconstruction}
Given $0<p, q\leq \infty$ and $\beta\in\mathbb{R}$. For any finite sequence $\{c_{n,j}\}_{n\in\mathbb{Z}^d,j\in J}$ of complex numbers we have
\begin{equation*}
\norm{\sum_{n\in\mathbb{Z}^d, j\in J} c_{n,j} |T_j|^{1/p-1/2}\eta_{n,j}}_{\dot{M}_{p,q}^\beta(\tilde{h})} \leq C\norm{\{c_{n,j}\}_{n\in\mathbb{Z}^d,j\in J}}_{\dot{m}_{p,q}^\beta (\tilde{h})}.
\end{equation*}
\end{mylemma}

\begin{proof}
Since finite sequences are dense in $\ell_q$ we may assume $\{c_{n,j'}\}_{n\in\mathbb{Z}^d,j'\in J}$ is finite. 
Let $f = \sum_{n\in\mathbb{Z}^d,j'\in J} c_{n,j'}\eta_{n,j'}$. For $j\in J$ we have
\begin{align}\label{coefficient1}
\norm{\phi_j^2(D)f}_{L_p} &= \norm{\phi_j^2(D) \left(\sum_{j'\in\widetilde{j}} \sum_{n\in\mathbb{Z}^d} c_{n,j'}\tilde{\eta}_{n,j'}\right)}_{L_p} \leq C\sum_{j'\in\widetilde{j}} \norm{\sum_{n\in\mathbb{Z}^d} c_{n,j'}\eta_{n,j'}}_{L_p},
\end{align}
where the last inequality follows since $\phi^2_j(D)$ satisfies \eqref{eq:bandlimited}. Now, since $\eta_{n,j'}$ satisfies the inequalities \eqref{eq:etaestimates1} and \eqref{eq:etaestimates2} we use the same technique as in the proof of Lemma \ref{lemma:stability} to obtain
\begin{equation}\label{eq:coefficient2}
\norm{\sum_{n\in\mathbb{Z}^d} c_{n,j'}\eta_{n,j'}}_{L_p} \leq C|T_{j'}|^{1/2-1/p} \norm{\{c_{n,j'}\}_{n\in\mathbb{Z}^d}}_{\ell_p}, \quad 0<p\leq \infty,
\end{equation}
uniformly in $j\in J$. Combining \eqref{coefficient1} and \eqref{eq:coefficient2} and rearranging terms we get
\begin{equation}\label{eq:coefficient3}
\norm{\phi_j^2(D)\left(f|T_{j'}|^{1/p-1/2}\right)}_{L_p} \leq C \sum_{j'\in \tilde{J}} \norm{\{c_{n,j'}\}_{n\in\mathbb{Z}^d}}_{\ell_p}.
\end{equation}
Using \eqref{eq:coefficient3} we obtain
\begin{align*}
\norm{\sum_{n\in\mathbb{Z}^d, j\in J} c_{n,j} |T_j|^{1/p-1/2}\eta_{n,j}}_{\dot{M}_{p,q}^\beta(\tilde{h})} 
&\leq C \left( \sum_{j\in J} \left( \tilde{h}(\xi_j)^{\beta} \sum_{j'\in \tilde{J}} \norm{\{c_{n,j'}\}_{n\in\mathbb{Z}^d}}_{\ell_p}\right)^q \right)^{1/q} \\
&\leq C\norm{\{c_{n,j}\}_{n\in\mathbb{Z}^d,j\in J}}_{\dot{m}_{p,q}^\beta (\tilde{h})}.
\end{align*}
This proves the lemma. 
\end{proof}

We can now verify that $\{\eta_{n,j}\}_{n\in\mathbb{Z}^d,j\in J}$ gives an atomic decomposition of the homogeneous decomposition spaces.
We define the coefficient operator $C: {\dot{M}_{p,q}^\beta (\tilde{h})} \to {\dot{m}_{p,q}^\beta (\tilde{h})}$ by $Cf = \left\{\ip{f,\eta_{n,j}^p}\right\}_{n,j}$ and the reconstruction operator $R: {\dot{m}_{p,q}^\beta (\tilde{h})} \to {\dot{M}_{p,q}^\beta (\tilde{h})}$ by $\{c_{n,j}\}_{n,j} \to \sum_{n,j} c_{n,j}|T_j|^{1/p-1/2}\eta_{n,j}$.

\begin{mytheorem}\label{thm:banachframe}
Given $0<p, q\leq\infty$ and $\beta \in \mathbb{R}$. Then the coefficient operator $C$ and the reconstruction operator $R$ are both bounded and makes $\dot{M}_{p,q}^\beta (\tilde{h})$ a retract of $\dot{m}_{p,q}^\beta (\tilde{h})$, i.e. $RC = \text{Id}_{\dot{M}_{p,q}^\beta (\tilde{h})}$. In particular, for $p\geq 1$, $\{\eta_{n,j}\}_{n\in\mathbb{Z}^d,j\in J}$ gives an atomic decomposition of the homogeneous decomposition spaces ${\dot{M}_{p,q}^\beta (\tilde{h})}$. 
\end{mytheorem}

\begin{proof}
It follows from Proposition \ref{prop:framecoeff} that the coefficient operator $C$ is a bounded linear operator and by extending Lemma \ref{lemma:reconstruction} to infinite sequences it can easily be verified that the reconstruction operator $R$ is a bounded linear operator. Thus we can illustrate the retract result of Theorem \ref{thm:banachframe} by the following commuting diagram. 
\[
\begin{tikzcd}
\dot{M}_{p,q}^\beta(\tilde{h}) \arrow{dr}{C} \arrow{rr}{\text{Id}_{\dot{M}^\beta_{p,q}(\tilde{h})}}&& \dot{M}_{p,q}^\beta(\tilde{h})\\
 & \dot{m}^\beta_{p,q}(\tilde{h}) \arrow{ur}{R} 
\end{tikzcd}
\]
\end{proof}

We end this section with an interesting corollary to the characterization \eqref{eq:equivnorm} given in Proposition \ref{prop:framecoeff}.
\begin{mycor}\label{korollar1}
We have the following (quasi)-norm equivalence on $\dot{M}_{p,q}^\beta (\tilde{h})$.
\begin{equation*}
\norm{f}_{\dot{M}_{p,q}^\beta (\tilde{h})} \asymp \inf\left\{\norm{\{c_{n,j}\}_{n,j}}_{\dot{m}_{p,q}^\beta (\tilde{h})} \; : \; f= \sum_{n,j} c_{n,j} |T_j|^{1/p-1/2}\eta_{n,j}\right\}.
\end{equation*}
\end{mycor}
The corollary shows that the canonical frame coefficients $\left\{\ip{f,\eta^p_{n,j}}\right\}_{n,j}$ are (up to a constant) the sparsest choice for expanding $f$ as in \eqref{eq:expanddistribution},
in the way that they minimize the $\dot{m}^\beta_{p,q}(\tilde{h})$-norm.

\section{Homogeneous Anisotropic \texorpdfstring{$\alpha$}--Modulation Spaces}\label{sec:examples}
In this section we provide an example of what we consider the natural extension of $\alpha$-modulation spaces to the homogeneous setting. We will need a hybrid regulation function as defined in Definition \ref{def:hfunction}. We use a quasi-distance $\d(\xi,\zeta) := |\xi-\zeta|_{\vec{a}}$ induced by the anisotropic norm $|\cdot|_{\vec{a}}$ to generate tight frames for homogeneous anisotropic Besov spaces and homogeneous anisotropic $\alpha$-modulation spaces. 

In the simple inhomogeneous setup on the real line, an $\alpha$-covering can easily be obtained from the knots $\pm n^\beta$, $n\in\mathbb{N}$, taking $\beta=1/(1-\alpha)$ for $0\leq \alpha<1$, while in the limiting (Besov) case $\alpha=1$, we simply use dyadic knots $\pm 2^j$, $j\in\mathbb{N}$. Now, in the Besov case, we add the low frequency knots $\pm 2^{-j}$, $j\in\mathbb{N}$, to obtain a full decomposition eventually yielding homogeneous Besov spaces. Notice that the low frequency knots can be considered the image under $\xi\rightarrow 1/\xi$ of the high frequency knots. We copy this process for the $\alpha$-covering obtaining low-frequency knots $\pm n^{-\beta}$, $n\in\mathbb{N}$, that can be seen to satisfy the geometric ``rule'' $|n^{-\beta}-(n+1)^{-\beta}|\asymp n^{-\beta(2-\alpha)}$, while the high-frequency knots satisfy $|(n+1)^\beta-n^\beta|\asymp n^{\alpha\beta}$. 

Inspired by these considerations, we define a general hybrid regulation function as $\tilde{h}_{\vec{a}}^\alpha(\xi) := \rho h_1(\xi) + (1-\rho)h_2(\xi)$, where $\rho$ satisfies \eqref{eq:rho}, $h_1(\xi) = |\xi|_{\vec{a}}^{2-\alpha}$ and $h_2(\xi) = |\xi|_{\vec{a}}^{\alpha}$, with $\alpha \in [0,1]$ fixed. 
Observe that the functions $h_1, h_2$ satisfy the conditions in \eqref{eq:h1}, \eqref{eq:h2} and they satisfy the condition of d-moderateness separately. So, according to Lemma \ref{lemma1}, $\tilde{h}_{\vec{a}}^\alpha$ is d-moderate. 
Note that any covering ball $B_1$ associated with $h_1$ satisfies the following geometric rule:
\begin{equation*}
\xi\in B_1 \Rightarrow |\xi|_{\vec{a}}^{(2-\alpha)d} \asymp |B_1|.
\end{equation*}
Whereas any covering ball $B_2$ associated with $h_2$ satisfies
\begin{equation*}
\xi \in B_2 \Rightarrow |\xi|_{\vec{a}}^{\alpha d} \asymp |B_2|.
\end{equation*}
In the transition zone between $h_1(\xi)$ and $h_2(\xi)$, that is, when $\tfrac{2}{3} < |\xi|_{\vec{a}} < \tfrac{4}{3}$, the volume of the covering balls $\widetilde{B}$ satisfy $|\widetilde{B}| \asymp 1$. 
With $|\cdot|_{\vec{a}}$ defined as in Definition \ref{def:aninorm},
we use Proposition \ref{thm:sdc} to conclude that $\tilde{h}_{\vec{a}}^\alpha$ determines one equivalence class of structured admissible coverings. 
Let us now consider different values of $\alpha \in [0,1]$ to see which spaces they give rise to. 

\subsection*{The Case \texorpdfstring{$\alpha = 1$}{alpha=1}}
For the underlying decomposition of $\mathbb{R}^d\backslash\{0\}$ we will use cubes and corridors. 
Given $\vec{a} = (a_1,\ldots,a_d)$ we define the cubes
\begin{equation*}
Q_j := \{\xi : |\xi_i| \leq 2^{ja_i}, \quad i=1,\ldots,d\}, \quad \text{for} \; j\in\mathbb{Z},
\end{equation*}
and the corridors $K_j = Q_j\backslash Q_{j-1}$ for $j\in\mathbb{Z}$. Let $E:= E_2^d\backslash E_1^d$ where $E_2 = \{\pm 1, \pm 2 \}$ and $E_1 := \{\pm 1\}$. For each $j\in\mathbb{Z}$ and $k\in E$ we define
\begin{equation*}
P_{j,k} := \{\xi\in\mathbb{R}^d\backslash\{0\} : \text{sgn}(\xi_i) = \text{sgn}(k_i), \; \text{and} \; (|k_i|-1)2^{j-1} \leq |\xi_i|^{1/a_i} \leq |k_i|2^{j-1}\}.
\end{equation*}
Then $K_j \subset \cup_{k\in E} P_{j,k}$ and the family $\{P_{j,k}\}_{j\in\mathbb{Z},k\in E}$ gives an anisotropic decomposition of $\mathbb{R}^d\backslash\{0\}$. 
The cubes $P_{j,k}$ can be generated by a family of affine transformations $\mathcal{T}$. Let $b_{j,k}$ be the center of the cube $P_{j,k}$, and define the function $B:E \to \mathbb{R}^{d\times d}$ by
\begin{equation*}
B(k) = \text{diag}(h(k)_1,\ldots,h(k)_d), \quad \text{where} \; h(k)_i 
= 
\begin{cases}
 2^{-(a_i+1)}& \text{if \; $|k_i| = 1$,} \\ 
 (1-2^{-a_i})/2& \text{if \; $|k_i| = 2$.}
 \end{cases}
\end{equation*}
Then the family $\mathcal{T} = \{T_{j,k}\}_{j\in\mathbb{Z},k\in E}$ given by $T_{j,k}\xi = D_{\vec{a}}(2^j)B(k)\xi + b_{j,k}$ generates the sets $P_{j,k}$. It can be verified that the covering system $\mathcal{T} = \{T_{j,k}\}_{j\in\mathbb{Z},k\in E}$ is equivalent to those used for the homogeneous anisotropic Besov spaces, see e.g. \cite[p. 223]{HansTriebel} for the inhomogeneous case. In fact $\dot{M}_{p,q}^\beta(\tilde{h}^{1}_{\vec{a}}, |\cdot|_{\vec{a}}) = \dot{B}_{p,q}^{\vec{a},\beta}$, so the tight frame $\{\eta_{n,j}\}_{n\in\mathbb{Z}^d,j\in J}$ from Section \ref{section:frames} gives an atomic decomposition of $\dot{B}_{p,q}^{\vec{a},\beta}$ and \eqref{eq:equivnorm} in Proposition \ref{prop:framecoeff} gives a characterization of the norm on $\dot{B}_{p,q}^{\vec{a},\beta}$. We refer to Bownik \cite{MR2179611} for a much more detailed study of such Besov-type spaces.

\subsection*{The Case \texorpdfstring{$0 \leq \alpha < 1$}{0<alpha<1}}
This is a more interesting case as we obtain new families of spaces, which we will call homogeneous anisotropic $\alpha$-modulation spaces. Since the assumptions of Proposition \ref{thm:sdc} are satisfied we conclude that the space $\dot{M}_{p,q}^\beta(\tilde{h}_{\vec{a}}^\alpha,|\cdot|_{\vec{a}})$ is well-defined in the sense that it only depends on the function $\tilde{h}_{\vec{a}}^\alpha$ up to equivalence of norms and is independent of the particular choice of BAPU. According to Proposition \ref{thm:sdc} there exists a structured admissible covering of $\mathbb{R}^d\backslash\{0\}$ associated with $\tilde{h}_{\vec{a}}^\alpha$. However, we do not (in general) know of any explicitly given structured covering. Nevertheless, we can assume that $\mathcal{T} = \{A_j \cdot + b_j\}_{j\in J}$ is one of the equivalent structured coverings given by Proposition \ref{thm:sdc}. Then the tight frame $\{\eta_{n,j}\}_{n\in\mathbb{Z}^d,j\in J}$ gives an atomic decomposition of $\dot{M}_{p,q}^{\vec{a},\beta,\alpha}$ and \eqref{eq:equivnorm} in Proposition \ref{prop:framecoeff} gives a characterization of the norm on $\dot{M}_{p,q}^{\vec{a},\beta,\alpha}$.
\appendix
\section{Proof of some technical lemmas}\label{sec:appA}
In this section we state some of the lemmas used throughout this paper.

\begin{mylemma}\label{lemma12}
Given $p \in (0,\infty]$, and an invertible affine transformation $T_j := A_j\cdot + b_j$. For a function $f\in L_p(\mathbb{R}^d)$ let $\hat{f}_j(\xi) := \hat{f}(T_j^{-1}\xi)$. Then
\begin{equation*}
\norm{f_j}_{L_p} = |T_j|^{1-1/p}\norm{f}_{L_p}, \quad 0<p\leq \infty. 
\end{equation*}
\end{mylemma}

\begin{proof}
If $T_j = A_j \cdot + b_j$, then $\hat{f}_j(\xi) = \hat{f}(A_j^{-1}(\xi - b_j))$. By a simple substitution we find that
\begin{align} \label{eq:lemma1}
f_j(x) = |T_j|\e^{ix\cdot b_j} f(A_j^T x).
\end{align}
Taking the $L_p$-norm of the expression in \eqref{eq:lemma1} and substituting once more we obtain
\begin{align*}
\norm{f_j}_{L_p} = |T_j| \left( \int_{\mathbb{R}^d} \left| \e^{ix\cdot b_j} f(A_j^Tx)\right|^p \d x \right)^{1/p} = |T_j|^{1-1/p}\norm{f}_{L_p}.
\end{align*}
\end{proof}

\begin{mylemma}\label{lemma:nikolskij}
Suppose $f \in S_0(\mathbb{R}^d)$ satisfies $\text{supp}(\hat{f}) \subset B_d(\xi_j,\delta\tilde{h}(\xi_j))$ for $j\in J$. Given an invertible affine transformation $T_j$, let $\hat{f}_j(\xi) := \hat{f}(T_j^{-1}\xi)$. Then, for $0<p\leq q \leq \infty$ 
\begin{equation*}
\norm{f_j}_{L_q} \leq C |T_j|^{1/p-1/q} \norm{f_j}_{L_p},
\end{equation*}
for some constant $C$ independent of $T_j$. 
\end{mylemma}

\begin{proof}
Using Lemma \ref{lemma12} and the Plancherel-Polya-Nikol'skij inequality \cite[Proposition 1.3.2]{HansTriebel} we obtain 
\begin{align*}
\norm{f_j}_{L_q} &= |T_j|^{1-1/q}\norm{f}_{L_q} \leq |T_j|^{1-1/q} C\norm{f}_{L_p} \\
& \leq C |T_j|^{1/p-1/q} \norm{f_j}_{L_p}.
\end{align*}
\end{proof}

\section{Proof of Proposition \ref{prop:complete} }\label{sec:appB}
Here we provide the proof of Proposition \ref{prop:complete}. We begin by recalling that any sequence of complex numbers satisfies
$\norm{\{a_k\}_k}_{\ell_{q_1}} \leq \norm{\{a_k\}_k}_{\ell_{q_0}}$ for $0<q_0\leq q_1 \leq \infty$. 
Let $\mathcal{Q} = \{Q_j\}_{j\in J}$ be a structured admissible covering, $0<p<\infty$ and $\beta \in \mathbb{R}$. An immediate consequence of the above yields the following embedding
\begin{equation}\label{eq:modembed}
\dot{M}_{p,q_0}^\beta(\tilde{h}) \hookrightarrow \dot{M}_{p,q_1}^\beta(\tilde{h}),
\end{equation}
for $0<q_0\leq q_1 \leq \infty$. 

\begin{proof}[Proof of Proposition \ref{prop:complete}]
We begin by proving the left hand-side of \eqref{eq:embeds}. According to \eqref{eq:modembed} we may assume $q<\infty$. 
Take $f\in \mathcal{S}_0(\mathbb{R}^d)$ and let $N > d/p$. We apply Hölder's inequality to obtain
\begin{align} \label{eq:s0embed}
\norm{f}^q_{\dot{M}_{p,q}^\beta (\tilde{h})} &= 
\sum_{j\in J} \tilde{h}(\xi_j)^{\beta q} \norm{\phi_j^2(D)f}_{L_p}^{q} 
\leq C \sum_{j\in J} |\xi_j|_{\vec{a}}^{\beta q} \norm{(1+|\cdot|)^N \mathcal{F}^{-1}\phi_j^2\hat{f}}_{L_\infty}^q \nonumber \\
&\leq C \sum_{j\in J} |\xi_j|_{\vec{a}}^{\beta q} \sum_{|\beta|\leq N} \norm{\mathcal{F}^{-1}\partial^\beta [\phi_j^2\hat{f}]}_{L_\infty}^q \nonumber \\
&\leq C \sum_{j\in J} |\xi_j|_{\vec{a}}^{\beta q} \sum_{|\beta|\leq N} \norm{\partial^\beta[\phi_j^2\hat{f}]}_{L_1}^q.
\end{align}
We consider the $L_1$-norm in \eqref{eq:s0embed} in several steps. First, using Leibniz's formula for $\partial^\beta \phi_j^2 \hat{f}$ we derive $\partial^\beta \phi_j^2 \hat{f} = \sum_{|\alpha+\eta| = |\beta|} C_{\alpha,\eta} \partial^\alpha \phi_j^2 \cdot \partial^\eta \hat{f}$, for certain constants $C_{\alpha,\eta}$. Consider first $\partial^\eta \hat{f}$. Since $f\in S_0(\mathbb{R}^d)$ we have by \eqref{eq:seminormS0},
	\begin{equation}\label{eq:partialf}
	|\partial^\eta \hat{f}(\xi)| \leq C_M \left( |\xi|^M + |\xi|^{-M}\right)^{-1} \norm{f}_M,
	\end{equation}
	for any $M \in \mathbb{N}$. Now, for $ \partial^\alpha \phi_j^2(\xi)$ we first recall that
	\begin{equation*}
		\partial^\alpha \phi_j^2(\xi)= \partial^\alpha \frac{\Phi^2(T_j^{-1}\xi)}{\sum_{k\in J} \Phi^2(T_{k}^{-1}\xi)}.
	\end{equation*}
	With this expression we may use the formula in \eqref{eq:Leib}. By a similar approach as in the proof of Proposition \ref{prop1}, we use the chain rule to obtain 		$\partial^\alpha \phi_j^2(\xi) =  \sum_{\mu: |\mu|=|\alpha|} p_{\mu} \partial^\mu \Phi^2$,
where $p_{\mu}$ are monomials of degree $|\alpha|$ in the entries of $A_j^{-1}$. 
Now, if $j\in J_2$ the entries in $A_j^{-1}$ are uniformly bounded (see the proof of Lemma \ref{lemma:konvergens}).
If $j\in J_1$ we recall that $h_1(\xi) \geq c_0|\xi|_{\vec{a}}^r$ for some $c_0 > 0$ and $r\geq 1$. Thus each entry in $A_j^{-1}$ is bounded by $|\xi_j|_{\vec{a}}^{-r}$.
Combining this and using \eqref{eq:anormJ1} it follows that for $L > \alpha_1 rd$,
\begin{equation}\label{eq:partialphi}
\left| \partial^\alpha \phi_j^2(\xi)\right| \leq C_\alpha \left(|\xi|^L + |\xi|^{-L}\right) \chi_{Q_j}(\xi).
\end{equation}
Combining \eqref{eq:partialf} and \eqref{eq:partialphi} with the expression in \eqref{eq:s0embed} we get
\begin{flalign*}
&\norm{f}^q_{\dot{M}_{p,q}^\beta(\tilde{h})} \\
&\leq C \sum_{j\in J} |\xi_j|_{\vec{a}}^{\beta q} \sum_{|\beta|\leq N} \left(\int_{\mathbb{R}^d} \left|\partial^\beta \phi_j^2\hat{f}(\xi)\right| \d \xi \right)^q \\
&\leq C \sum_{j\in J} |\xi_j|_{\vec{a}}^{\beta q} \sum_{|\alpha|,|\eta| \leq N} \int_{\mathbb{R}^d} \left( \left| \partial^\alpha \phi_j^2(\xi) \cdot \partial^\eta \hat {f} (\xi) \right| \right)^q \d \xi \\
&\leq C \sum_{j\in J} |\xi_j|_{\vec{a}}^{\beta q} \int_{\mathbb{R}^d} \left[ \left( |\xi|^L+|\xi|^{-L}\right) \chi_{Q_j}(\xi) \cdot \norm{f}_M \left(|\xi|^M+|\xi|^{-M}\right)^{-1} \right]^q \d \xi \\
&\leq C \norm{f}_M^q \sum_{j\in J} |\xi_j|_{\vec{a}}^{\beta q} \sup_{\xi \in Q_j} \left[ (|\xi|^M+|\xi|^{-M})^{-1/2} \int_{\mathbb{R}^d} (|\xi|^L+|\xi|^{-L})(|\xi|^M+|\xi|^{-M})^{-1/2} \right]^q \d \xi \\
&\leq C \norm{f}_M^q \sum_{j\in J} |\xi_j|_{\vec{a}}^{\beta q} \sup_{\xi \in Q_j} \left[ (|\xi|^M+|\xi|^{-M})^{-1/2} \right]^q,
\end{flalign*}
for $M$ chosen such that $\frac{1}{2}M > L + d + 1$. 
We next consider two cases. 
Recall the estimates in \eqref{eq:anormJ2}, \eqref{eq:anormJ1}, and note that, since $h_1,h_2$ are continuous functions we can adjust the coefficients $c_1,c_2$ appearing in \eqref{eq:anormJ2} such that these inequalities holds for all $|\xi_j|_{\vec{a}}$ with $j\in J_2$. Thus, for $j\in J_2$ we increase the value of $M$, if needed, such that $\frac{1}{2}M > \alpha_1 (- \beta-(d+1)/q)$. Then
\begin{align}\label{eq:embedJ2}
	\norm{f}^q_{\dot{M}_{p,q}^\beta(\tilde{h})} &\leq C \norm{f}_M^q \sum_{j\in J_2} |\xi_j|_{\vec{a}}^{\beta q} \sup_{\xi\in Q_j} \left[(|\xi|^M + |\xi|^{-M})^{-1/2}\right]^q \nonumber \\
	&\leq C \norm{f}_M^q \sum_{j\in J_2} |\xi_j|_{\vec{a}}^{\beta q} |\xi_j|_{\vec{a}}^{(-\beta - (d+1)/q)q} \nonumber \\
	&\leq C \norm{f}_M^q \sum_{j\in J_2} |\xi_j|_{\vec{a}}^{-(d+1)} \nonumber \\
	&\leq C \norm{f}_M^q.
\end{align}
Now, let $j\in J_1$. By increasing the value of $M$, if needed, we can assume that $\frac{1}{2}M > \alpha_2 (rd/q -\beta)$. Then
\begin{align}\label{eq:embedJ1}
\norm{f}^q_{\dot{M}_{p,q}^\beta(\tilde{h})} &\leq C \norm{f}_M^q \sum_{j\in J_1} |\xi_j|_{\vec{a}}^{\beta q} \sup_{\xi\in Q_j} \left[(|\xi|^M + |\xi|^{-M})^{-1/2}\right]^q \nonumber \\
&\leq C \norm{f}_M^q \sum_{j\in J_1} |\xi_j|_{\vec{a}}^{\beta q} |\xi_j|_{\vec{a}}^{(rd/q-\beta)q} \nonumber \\
&\leq C \norm{f}_M^q \sum_{j\in J_1} |\xi_j|_{\vec{a}}^{rd} \nonumber \\
&\leq C \norm{f}_M^q.
\end{align}
By \eqref{eq:embedJ2} and \eqref{eq:embedJ1} we conclude that $\mathcal{S}_0 \hookrightarrow \dot{M}_{p,q}^\beta(\tilde{h})$ for any $M$ satisfying $\frac{1}{2}M > \max\{L+d+1, \alpha_1(-\beta-(d+1)/q), \alpha_2 (rd/q - \beta)\}$. 
	
Let us now prove the right-hand side of \eqref{eq:embeds} with $q=\infty$. Let $f\in \dot{M}_{p,\infty}^\beta(\tilde{h})$ and $\theta \in \mathcal{S}(\mathbb{R}^d)$. Then,
\begin{align*}
\ip{f,\theta} &= \sum_{j\in J} \ip{\check{\phi}_j^2*f,\theta}
= \sum_{j\in J} \ip{\phi^2_j\hat{f},\hat{\theta}} = \sum_{j\in J} \ip{\phi_j^2\hat{f},\tilde{\phi}_j^{2}\hat{\theta}} \\
&= \sum_{j\in J} \ip{\mathcal{F}^{-1}(\phi^2_j\hat{f}),\mathcal{F}^{-1}(\tilde{\phi}_j^{2}\hat{\theta})},
	\end{align*}
where we used that $\phi_j^2\tilde{\phi}_j^{2} = \phi_j^2$. 
Since $\text{supp}(\phi_j)$ is contained in $B_d(\xi_j,\delta \tilde{h}(\xi_j)), j\in J$ we use Hölder's inequality and Lemma \ref{lemma:nikolskij} with $q=\infty$ to obtain
	\begin{align*}
|\ip{f,\theta} |
&\leq \sum_{j\in J} \norm{\mathcal{F}^{-1}(\phi^2_j\hat{f})}_{L_\infty} \norm{\mathcal{F}^{-1}(\tilde{\phi}_j^{2}\hat{\theta})}_{L_1} \\
&\leq C \sum_{j\in J} \tilde{h}(\xi_j)^{d/p} \norm{\mathcal{F}^{-1}(\phi_j^2\hat{f})}_{L_p} \norm{\mathcal{F}^{-1}(\tilde{\phi}_j^{2}\hat{\theta})}_{L_1} \\
&\leq C\norm{f}_{\dot{M}_{p,\infty}^\beta (\tilde{h})} \sum_{j\in J }\tilde{h}(\xi_j)^{-\beta}\tilde{h}(\xi_j)^{d/p} \norm{\mathcal{F}^{-1}(\tilde{\phi}_j^{2}\hat{\theta})}_{L_1} \\	&\leq C \norm{f}_{\dot{M}_{p,\infty}^\beta (\tilde{h})} \norm{\theta}_{\dot{M}_{1,1}^{-\beta + d/p}(\tilde{h})}.
	\end{align*}
	Using that $\mathcal{S}_0(\mathbb{R}^d) \hookrightarrow \dot{M}_{1,1}^{-\beta + d/p}(\tilde{h})$, we obtain
	\begin{equation*}
|\ip{f,\theta}| \leq C \norm{f}_{\dot{M}_{p,\infty}^\beta(\tilde{h})} \norm{\theta}_M,
	\end{equation*}
for sufficiently large $M$. This proves that $\dot{M}_{p,\infty}^\beta(\tilde{h})$ is continuously embedded in $\mathcal{S}'\backslash\mathcal{P}$. Combining with \eqref{eq:modembed} shows part (1) of Proposition \ref{prop:complete}.
	Let us now show that $\dot{M}_{p,q}^\beta(\tilde{h})$ is complete.  
	Let $\{f_n\}_{n\in\mathbb{N}}$ be a Cauchy sequence in $\dot{M}_{p,q}^\beta(\tilde{h})$. As a consequence of the above, $f_n$ converges in $\mathcal{S}'\backslash\mathcal{P}$ to an element $f\in\mathcal{S}'\backslash\mathcal{P}$. Applying Fatou's Lemma twice we obtain
	\begin{align*}
		\norm{f-f_n}_{\dot{M}_{p,q}^\beta(\tilde{h})} &= \left( \sum_{j\in J} \left( \tilde{h}(\xi_j)^\beta \norm{\phi_j^2(D)(f-f_n)}_{L_p}\right)^q\right)^{1/q} \\
		&\leq \left[ \liminf_{m\to\infty} \sum_{j\in J} \left( \tilde{h}(\xi_j)^\beta \norm{\phi_j^2(D)(f_m-f_n)}_{L_p}\right)^q\right]^{1/q} \\
		&\leq \liminf_{m\to\infty} \left[ \sum_{j\in J} \left( \tilde{h}(\xi_j)^\beta \norm{\phi_j^2(D)(f_m-f_n)}_{L_p}\right)^q\right]^{1/q} \\
		&= \liminf_{m\to\infty} \norm{f_m-f_n}_{\dot{M}_{p,q}^\beta(\tilde{h})}.
	\end{align*}
This shows that $\norm{f-f_n}_{\dot{M}_{p,q}^\beta(\tilde{h})} \to 0$ as $n\to\infty$. Moreover, since
\begin{align*}
\norm{f}_{\dot{M}_{p,q}^\beta(\tilde{h})} &= \norm{f-f_n+f_n}_{\dot{M}_{p,q}^\beta(\tilde{h})} \\
&\leq C \left( \norm{f-f_n}_{\dot{M}_{p,q}^\beta(\tilde{h})} + \norm{f_n}_{\dot{M}_{p,q}^\beta(\tilde{h})}\right) < \infty,
\end{align*}
then $f$ belongs to $\dot{M}_{p,q}^\beta(\tilde{h})$. This proves that $\dot{M}_{p,q}^\beta(\tilde{h})$ is complete.
	
	We finally prove that $S_0(\mathbb{R}^d)$ is dense in $\dot{M}_{p,q}^\beta(\tilde{h})$ if $0<p<\infty$ and $0<q<\infty$. Let $f\in \dot{M}_{p,q}^\beta(\tilde{h})$ and set 		$f_N(x) = \sum_{j=-N}^N \mathcal{F}^{-1}\phi^2_j\mathcal{F}f$.
	Then $f_N \in \dot{M}_{p,q}^\beta(\tilde{h})$ and consequently,
	\begin{equation}\label{eq:dense}
	\norm{f-f_N}_{\dot{M}_{p,q}^\beta(\tilde{h})} \leq c  \left( \sum_{|j|>N}^\infty \left( \tilde{h}(\xi_j)^\beta \norm{\phi_j^2(D)f}_{L_p}\right)^q\right)^{1/q}.
	\end{equation}
	By Lebesgue's dominated convergence theorem, the right hand side of \eqref{eq:dense} tends to zero as $N\to\infty$. So $f_N$ approximates $f$ in $\dot{M}_{p,q}^\beta(\tilde{h})$, and if $f_N \in \mathcal{S}_0$ for all $N\in\mathbb{N}$ we have shown that any function in $\dot{M}_{p,q}^\beta(\tilde{h})$ can be approximated by functions in $\mathcal{S}_0(\mathbb{R}^d)$. Otherwise, we need to find $(f_N)_\delta \in S_0(\mathbb{R}^d)$ which approximates $f_N$ in $\dot{M}_{p,q}^\beta(\tilde{h})$. This can be done using convolution with an approximate identity. 
	Suppose $\tilde{\phi} \in \mathcal{S}(\mathbb{R}^d)$ satisfies $\int_{\mathbb{R}^d} \tilde{\phi}(\xi) = 1$ and define $\tilde{\phi}_\delta(\xi) := \delta^{-d}\tilde{\phi}(d/\xi)$. Moreover, for $\delta>0$ we define a smooth ramp function $\rho_\delta$ as 
	\begin{equation*}
		\rho_\delta(\xi) = \begin{cases}
			0& \text{if $|\xi|_{\vec{a}} < \delta$}, \\ 
			1& \text{if $2\delta \leq |\xi|_{\vec{a}}$}.
		\end{cases}
	\end{equation*} 
	Let $\phi_\delta(\xi) := \rho_\delta(\xi) \tilde{\phi}_\delta(\xi)$ and define
	\begin{equation*}
		(f_N)_\delta := \mathcal{F}^{-1} \{\rho_\delta * \tilde{\phi}_\delta * \mathcal{F}(f_N)\} = \mathcal{F}^{-1} \{\tilde{\phi}_\delta * (\rho_\delta * \mathcal{F}(f_N))\}.
	\end{equation*}
	Then $(f_N)_\delta \in S_0(\mathbb{R}^d)$ approximates $f_N$ in $L_p(Q_j)$. Since this is also an approximation of $f_N$ in $\dot{M}_{p,q}^\beta(\tilde{h})$ we have shown that $S_0(\mathbb{R}^d)$ is dense in $\dot{M}_{p,q}^\beta(\tilde{h})$.
\end{proof}
\bibliographystyle{abbrv}

\end{document}